\documentclass[10pt]{amsart}

\newtheorem{theorem}{Theorem}[section]
\newtheorem{lemma}[theorem]{Lemma}
\newtheorem{corollary}[theorem]{Corollary}

\theoremstyle{definition}

\theoremstyle{remark}

\numberwithin{equation}{section}

\def\phi{\varphi}

\def\A{\mathbb A}
\def\N{\mathbb N}

\def\mm{{\bf m}}
\newcommand{\rar}{\rightarrow}

\usepackage{amscd}
\usepackage{graphicx}
\usepackage[all]{xy}

\title[Minimal Free Resolution for monomial curves in $\mathbb{A}^{4}$]
{Minimal graded free resolutions for monomial curves in $\mathbb{A}^{4}$ defined by almost arithmetic sequences}

\author[Roy]{Achintya Kumar Roy}
\address{Department of Mathematics, Shyamsundar College, Burdwan 713424, India.} \email{achintya.roy@gmail.com}
\author[Sengupta]{Indranath Sengupta}
\address{Discipline of Mathematics, IIT Gandhinagar, VGEC Campus, Visat-Gandhinagar Highway, Chandkheda, Ahmedabad, Gujarat 382424, INDIA.} 
\email{indranathsg@iitgn.ac.in}
\thanks{The second author is the corresponding author, who is supported by the the IITGN internal project IP/IITGN/MATH/IS/201415-13.}
\author[Tripathi]{Gaurab Tripathi}
\address{Department of Mathematics, Jadavpur University, Kolkata,
WB 700 032, India.} \email{gelatinx@gmail.com}
\thanks{The third author thanks CSIR for the Senior Research Fellowship.}

\date{}



\begin{document}

\begin{abstract}
Let $\mm=(m_0,m_1,m_2,n)$ be an almost arithmetic sequence, i.e., a
sequence of positive integers with ${\rm gcd}(m_0,m_1,m_2,n) = 1$, such that 
$m_0<m_1<m_2$ form an arithmetic progression, $n$ is arbitrary and they 
minimally generate the numerical semigroup $\Gamma = m_0\N + m_1\N + m_2\N + n\N$. Let 
$k$ be a field. The homogeneous coordinate ring $k[\Gamma]$ of the affine monomial
curve parametrically defined by $X_0=t^{m_0},X_{1}=t^{m_1},X_2=t^{m_2},Y=t^{n}$ is
a graded $R$-module, where $R$ is the polynomial ring $k[X_0,X_1,X_2, Y]$ with the 
grading $\deg{X_i}:=m_i, \deg{Y}:=n$. In this paper, we construct a minimal graded 
free resolution for $k[\Gamma]$. 
\end{abstract}

\maketitle

\section{Introduction}
Monomial curves in affine and projective spaces are interesting because of the 
correspondence between geometry of the curves and arithmetic of the 
numerical semigroups defining them. Several authors have studied these curves, 
and semigroup rings in general from the viewpoints of algebra and geometry; \cite{herzog}, \cite{cowsik}, 
\cite{GNW}, \cite{bresinski}, \cite{patil1}, \cite{patil2}, \cite{patseng}, \cite{malooseng}.
\medskip

Construction of an explicit minimal free resolution of a finitely generated $k$-algebra 
is a difficult problem in general. This problem has been studied extensively for the 
homogeneous coordinate ring of an affine monomial curve, with good conditions on the 
defining semigroups; \cite{seng2}, \cite{hip1}, \cite{hip2}. It was conjectured in \cite{hip1} 
and later proved in \cite{hip2} that an affine monomial curve defined by an arithmetic sequence 
of positive integers $m_{0} < \cdots < m_{e}$ has the interesting property that the total 
Betti numbers of its homogeneous coordinate ring depend only on the integers $e$ and $a$, 
where $m_{0}$ is congruent to $a$ modulo $e$.  As a consequence, a proof was obtained for 
the periodicity conjecture of Herzog and Srinivasan on the eventual periodicity under translation 
of the Betti numbers for the monomial curves defined by an arithmetic sequence. Formulas involving 
Castelnuovo-Mumford regularity and the Frobenius number were derived as corollaries in \cite{hip2}. 
The periodicity conjecture for monomial curves has been settled recently by Thanh Vu in \cite{vu}.
\medskip

Let $\mm=(m_0,m_1,m_2,n)$ be an almost arithmetic sequence, i.e., a
sequence of positive integers with ${\rm gcd}(m_0,m_1,m_2,n) = 1$, such that 
$m_0<m_1<m_2$ form an arithmetic progression, $n$ is arbitrary and they 
minimally generate the numerical semigroup $\Gamma = m_0\N + m_1\N + m_2\N + n\N$. 
Let $k$ denote an arbitrary field and $R$ denote the polynomial ring $k[X_0,X_1,X_2, Y]$. 
Consider the $k$-algebra homomorphism $\phi:R\rar k[t]$ given by 
$\phi(X_0)=t^{m_0}$, $\phi(X_1)=t^{m_1}$, $\phi(X_2)=t^{m_2}$, $\phi(Y)=t^{n}$. 
Then, the ideal $\mathfrak{p}:=\ker{\phi}\subset R$ is the defining ideal of height $3$ and 
$k[\Gamma]:=k[t^{m_0},t^{m_1},t^{m_2},t^{n}]\simeq R/\mathfrak{p}$ is the coordinate 
ring of dimension $1$ of the monomial curve in $\A_k^{4}$, given by the parametrization 
$X_0=t^{m_0}$, $X_1=t^{m_1}$, $X_2=t^{m_2}$, $Y=t^{n}$. It is well known that $\mathfrak{p}$ is minimally 
generated by binomials. Moreover, $\mathfrak{p}$ is a homogeneous ideal and $k[\Gamma]$ is the homogeneous coordinate 
ring with respect to the gradation $\deg{X_i}:=m_i, \deg{Y}:=n$. Henceforth, the words homogeneous and graded will be used 
for this weighted gradation. In this paper, we construct an explicit minimal graded free 
resolution for $k[\Gamma]$. We will make use of the explicit description 
of the Gr\"{o}bner basis of the defining ideal $\mathfrak{p}$ given in \cite{seng1}. We will retain all the notations 
that were introduced in \cite{patsingh} and later used in \cite{patil2}, \cite{patseng}, \cite{seng1} and \cite{seng2} 
for the sake of convenience. 
\medskip

This work is in the same vein as that followed in \cite{seng2}, \cite{hip1}, \cite{hip2}. It generalizes the results 
proved in \cite{seng2}. While working out explicit minimal free resolutions, the foremost 
important theme in this paper is to exhibit a similar pattern that we proved in \cite{hip2} and that is, among an 
infinite family of monomial curves $\mathbb{A}^{4}$  defined by almost arithmetic sequences there are only finitely 
many, in fact $8$ distinct values of Betti numbers for these curves. The main theorem we prove is the following:

\begin{theorem}
Let $k$ denote an arbitrary field, $R$ the polynomial ring $k[X_{0},X_{1},X_{2}, Y]$ and 
$\mm=(m_0,m_1,m_2,n)$ an almost arithmetic sequence. Let $\mathfrak{p}$ be the 
homogeneous defining ideal and $k[\Gamma]:= R/\mathfrak{p}$ the homogeneous coordinate ring 
of the monomial curve in $\A_k^{4}$, defined by the parametrization $X_0=t^{m_0}$, $X_1=t^{m_1}$, 
$X_2=t^{m_2}$, $Y=t^{n}$. A minimal free resolution for the homogeneous coordinate ring $k[\Gamma]$ 
is given by 
$$0\longrightarrow R^{\beta_{2}}\longrightarrow R^{\beta_{1}}\longrightarrow R^{\beta_{0}}\longrightarrow R\longrightarrow k[\Gamma]\longrightarrow 0$$
where $\beta_{0}$, $\beta_{1}$ and $\beta_{2}$ are the total Betti numbers given below in tabular form as the triplet 
$[\beta_{0}, \beta_{1}, \beta_{2}]$. 
\medskip

\begin{tabular}{|c|c|c|}
\hline
\textbf{Cases} & $r=r'=1$ & $r=1,r'=2$\\
\hline
$W\neq\emptyset$ 
& 
\begin{tabular}{l|l}
$[4,5,2]$ & if\, $\mu=0$\\
$[6,8,3]$ & if\, $\mu\neq 0$,$\lambda=1$\\
$[6,8,3]$ & if\, $\mu\neq 0$,$\lambda\neq 1$,$q'=0$\\
$[6,9,4]$ & if\, $\mu\neq 0$,$\lambda\neq 1$,$q'\neq 0$\\
{} & {}\\
{} & {}\\
\end{tabular}
&
\begin{tabular}{l|l}
$[4,6,3]$ & if\, $\mu=0,q-q'\neq 1$\\
$[4,5,2]$ & if\, $\mu=0,q-q'=1$\\
$[5,6,2]$ & if\, $\mu\neq 0$, $\lambda\neq1$, $q-q'=1$\\
$[5,6,2]$ & if\, $\mu\neq 0$, $\lambda=1$, $q-q'\neq1$\\
$[5,5,1]$ & if\, $\mu\neq 0$, $\lambda=1$, $q-q'=1$\\
$[5,7,3]$ & if\, $\mu\neq 0$, $\lambda\neq 1$, $q-q'\neq 1$ 
\end{tabular}\\
\hline
$W=\emptyset$& 
\hspace*{-1.4in}$[4,5,2]$ & \hspace*{-1.75in}$[4,5,2]$ \\
\hline
\end{tabular}
\medskip

\begin{tabular}{|c|c|c|}
\hline
\textbf{Cases} & $r=2,r'=1$ & $r=r'=2$\\
\hline
$W\neq \emptyset$ & 
\begin{tabular}{l|l}
$[5,6,2]$ & if\,  $\mu=0,q'\neq 0$\\ 
$[5,5,1]$ & if\, $\mu=0,q'=0$\\
$[5,6,2]$ & if\, $\mu\neq 0, q'= 0$\\
$[4,6,3]$ & if\, $\mu\neq0$, $q'=q$\\
$[5,7,3]$ & if\, $\mu\neq 0$, $q'\neq q$, $q'\neq 0$
\end{tabular}
&
\begin{tabular}{l|l}
$[4,5,2]$ & if\, $\mu\neq 0$\\
$[3,3,1]$ & if\, $\mu=0$\\
{} & {}\\
{} & {}\\
{} & {}\\
\end{tabular}\\
\hline
$W=\emptyset$ & 
\hspace*{-1.55in}$[3,3,1]$ & \hspace*{-0.6in}$[3,3,1]$\\
\hline
\end{tabular}
\medskip

The set $W = [u-z, u-1] \times [v-w, v-1]$ and the integers 
$u, v, w, z, r, r', q, q', \mu, \lambda, \nu$ are as defined in 
{\rm [Lemma(2.1), \mbox{\cite{seng1}}]} and {\rm [Lemma(2.2), \mbox{\cite{seng1}}]}.
\end{theorem}

\section{Preliminaries}
We now discuss in brief the strategy for proving Theorem 1.1. We first begin with the 
generators of the ideal $\mathfrak{p}$ given in \cite{patsingh}. This set of generators 
denoted by $\mathcal{G}$ is known as the Patil-Singh generators. We use the same 
notations introduced in \cite{patsingh}.
\medskip

\noindent Let $\xi_{11}=X_{1}^2-X_{0}X_{2}$\,.
\bigskip

\noindent For $i\in \,[0,2-r]$, \,\,let 
\,$\phi_{i} \,:= X_{r+i}X_{2}^{q} - X_{0}^{\lambda - 1}X_{i}Y^{w}$\,. 
\bigskip

\noindent For $j\in [0,2-r']$, \,let 
\,$\psi_{j}:= X_{r'+j}X_{2}^{q'}Y^{v-w} - X_{0}^{\nu - 1}X_{j}$\,.
\bigskip

\noindent Let \,$\theta := 
\begin{cases}
Y^{v} - X_{0}^{\mu}X_{r-r'}X_{2}^{q-q'} & {\rm if} \quad r'<r,\\ 
Y^{v} - X_{0}^{\mu}X_{2+r-r'}X_{2}^{q-q'-1} & {\rm if} \quad r'\geq r\,.
\end{cases}$
\bigskip

\noindent It is proved in [\cite{patsingh}; (4.5)] that the set 
$$\mathcal{G} := \{\xi_{11}\}\cup \{\phi_{i}\mid i\in [0, 2-r]\} \cup 
\{\psi_{j}\mid j\in [0, 2-r']\}\cup \{\theta\}$$ 
forms a binomial set of generators for $\mathfrak{p}$. It is not 
always minimal. A subset of this is given in \cite{patil2}, which is a minimal generating 
set for $\mathfrak{p}$. We, however, choose to work with the bigger set $\mathcal{G}$ 
because it is a Gr\"{o}bner basis for $\mathfrak{p}$ with respect to the graded reverse-lexicographic 
monomial order, see \cite{seng1}.
\medskip

The computation with the $S$-polynomials done in \cite{seng1} can be used together with Schreyer's theorem 
to write down the first syzygy for $\mathfrak{p}$. Let us recall Schreyer's theorem first.

\begin{theorem}\label{schreyer}{\it 
Let $K$ be a field, $K[X_{1}, \ldots, X_{n}]$ 
be the polynomial ring and $I$ be an ideal in $K[X_{1}, \ldots, X_{n}]$. 
Let $G := \{g_{1}, \ldots, g_{t}\}$ be an ordered set of generators 
for $I$, which is a Gr\"obner basis, with respect to some fixed monomial 
order on $K[X_{1}, \ldots, X_{n}]$. Let 
${\rm Syz}(g_{1}, \ldots, g_{t}) := \{(a_{1}, \ldots, a_{t})\in R^{t} \mid 
\sum_{i=1}^{t}a_{i}g_{i} = 0\}$. 
Suppose that, for $\,i\neq j$, 
$$ S(g_{i}, g_{j})  =  a_{j}g_{i} - a_{i}g_{j} 
= \sum_{k=1}^{t}h_{k}g_{k}\, \longrightarrow_{G} 0\,.$$
Then, the 
$\,t$-tuples 
$$\left( \begin{array}{lclclcl}
h_{1}\,, & \cdots \,, & h_{i} - a_{j}\,,  & \cdots \,, 
& h_{j} + a_{i}\,,  & \cdots \,, & h_{t}
\end{array}\right)\,$$ 
generate $\,{\rm Syz}(g_{1}, \ldots, g_{t})$.
}
\end{theorem}

\proof See Chapter 5, Theorem 3.2 in \cite{cox}.\qed
\bigskip

\noindent Our strategy would be to find a resolution of $k[\Gamma]$ which may not be minimal. The following lemmas will be useful to 
remove redundancies and extract the minimal free resolution of the ideal $k[\Gamma]$, 
which sits as a direct summand of the first resolution. 

\begin{lemma}\label{homo1}{\it 
Let 
$$R^{a_{1}}\stackrel{A_{1}}{\longrightarrow} R^{a_{2}}\stackrel{A_{2}}{\longrightarrow} R^{a_{3}}$$ 
be an exact sequence of free modules. Let $Q_{1}$, $Q_{2}$, $Q_{3}$ be invertible matrices of sizes 
$a_{1}$, $a_{2}$, $a_{3}$ respectively. Then, 
$$R^{a_{1}}\stackrel{Q_{2}^{-1}A_{1}Q_{1}}{\longrightarrow} R^{a_{2}}\stackrel{Q_{3}A_{2}Q_{2}}{\longrightarrow} R^{a_{3}}$$
is also an exact sequence of free modules.
}
\end{lemma}

\begin{proof}
The following diagram is a commutative diagram is free modules and the vertical maps are isomorphisms:
$$
\xymatrix{
R^{a_{1}} \ar[r]^{A_{1}}& R^{a_{2}} \ar[r]^{A_{2}} & R^{a_{3}} \\
R^{a_{1}} \ar[u]^{Q_{1}}\ar[r]^{Q_{2}^{-1}A_{1}Q_{1}} & R^{a_{2}} \ar[u]_{Q_{2}}\ar[r]^{Q_{3}A_{2}Q_{2}} & R^{a_{3}}\ar[u]_{Q_{3}}
} 
$$
Therefore, 
$R^{a_{1}}\stackrel{Q_{2}^{-1}A_{1}Q_{1}}{\longrightarrow} R^{a_{2}}\stackrel{Q_{3}A_{2}Q_{2}}{\longrightarrow} R^{a_{3}}$ 
is exact since 
$R^{a_{1}}\stackrel{A_{1}}{\longrightarrow} R^{a_{2}}\stackrel{A_{2}}{\longrightarrow} R^{a_{3}}$ 
is exact.
\end{proof}

\begin{corollary}\label{homo2}{\it 
Let 
$$R^{a_{1}}\stackrel{C}{\longrightarrow} R^{a_{2}}\stackrel{B}{\longrightarrow} R^{a_{3}}\stackrel{A}{\longrightarrow} R^{a_{4}}$$ 
be an exact sequence of free modules. Let $P_{1}$, $P_{2}$, $P_{3}$ be invertible matrices of sizes 
$a_{1}$, $a_{2}$, $a_{3}$ respectively. Then, 
$$R^{a_{1}}\stackrel{P_{2}^{-1}CP_{1}}{\longrightarrow} R^{a_{2}}\stackrel{P_{3}BP_{2}}{\longrightarrow} R^{a_{3}}\stackrel{AP_{3}^{-1}}{\longrightarrow} R^{a_{4}}$$
is also an exact sequence of free modules.
}
\end{corollary}

\begin{proof}
Consider the sequence $R^{a_{1}}\stackrel{C}{\longrightarrow} R^{a_{2}}\stackrel{B}{\longrightarrow} R^{a_{3}}$. If we take $Q_{1}=P_{1},Q_{2}=P_{2}$ and $Q_{3}=I$ and apply Lemma \ref{homo1}, we get that the sequence 
$R^{a_{1}}\stackrel{P_{2}^{-1}CP_{1}}{\longrightarrow} R^{a_{2}}\stackrel{BP_{2}}{\longrightarrow} R^{a_{3}}$ is exact. 
We further note that the entire sequence\, $R^{a_{1}}\stackrel{P_{2}^{-1}CP_{1}}{\longrightarrow} R^{a_{2}}\stackrel{BP_{2}}{\longrightarrow} R^{a_{3}}\stackrel{A}{\longrightarrow} R^{a_{4}}$ is exact as well, since ${\rm Im}(B)= {\rm Im}(BP_{2})$ and $P_{2}$ is invertible. Let us 
now consider the sequence $R^{a_{2}}\stackrel{BP_{2}}{\longrightarrow} R^{a_{3}}\stackrel{A}{\longrightarrow} R^{a_{4}}$. We take 
$Q_{1}=Q_{3}=I$, $Q_{2}= P_{3}^{-1}$ and apply Lemma \ref{homo1} to arrive at our conclusion.
\end{proof}

\begin{lemma}\label{hom3}
{\it 
Let 
$$\cdots \longrightarrow R^{\beta_{n+1}}\stackrel{A_{n+1}}{\longrightarrow} R^{\beta_{n}}\stackrel{A_{n}}\longrightarrow R^{\beta_{n-1}}\stackrel{A_{n-1}}\longrightarrow R^{\beta_{n-2}} \longrightarrow \cdots$$ 
be an exact sequence of free R modules. 
Let $a_{ij}$ denote the $(i,j)$-th entry of $A_{n}$. Suppose that $a_{lm}=1$ for some l and m, 
$a_{li}=0$ for $i\neq m$ and $a_{jm}=0$ for $j\neq l$. Let $A_{n+1}^{'}$ 
be the matrix obtained by deleting the m-th row from $A_{n+1}$, 
$A_{n-1}^{'}$ the matrix obtained by deleting the l-th column from $A_{n-1}$ and 
$A_{n}^{'}$ the matrix obtained by deleting the l-th row and m-th column 
from $A_{n}$. Then, the sequence 
$$\cdots \longrightarrow R^{\beta_{n+1}}\stackrel{A_{n+1}^{'}}{\longrightarrow} R^{\beta_{n}-1}\stackrel{A_{n}^{'}}
\longrightarrow R^{\beta_{n-1}-1}\stackrel{A_{n-1}^{'}}\longrightarrow R^{\beta_{n-2}} \longrightarrow \cdots$$ 
is exact.
}
\end{lemma}

\begin{proof}
The fact that the latter sequence is a complex is self evident. We need to prove its exactness. 
By the previous lemma we may assume that $l=m=1$, for we choose elementary matrices to permute rows and 
columns and these matrices are always invertible. Now, due to exactness of the first complex we have 
$A_{n-1}A_{n}=0.$ This implies that the first column of $A_{n-1}=0$, which implies that 
${\rm Im}(A_{n-1}) = {\rm Im}(A_{n-1}^{'})$. Therefore, the right exactness of $A_{n+1}$ is preserved. 
By a similar argument we can prove that the left exactness of $A_{n+1}^{'}$ is preserved.
\medskip

Let $\left({\bf\underline x}\right)$ denote a tuple with entries from $R$. 
If $\left({\bf\underline x}\right)\in {\rm ker}(A_{n}^{'})$, then $\left(0, {\bf\underline x}\right)\in {\rm ker}(A_{n})$. 
There exists $\left({\bf\underline y}\right)\in R^{\beta_{n+1}}$ such that 
$A_{n-1}\left({\bf\underline y}\right) = \left(0, {\bf\underline x}\right)$. It follows that 
$A_{n-1}^{'}\left({\bf\underline y}\right) = \left({\bf\underline x}\right)$, proving the left exactness of 
$A_{n}^{'}$. By a similar argument we can prove the right exactness of $A_{n}^{'}$.
\end{proof}
\medskip

\noindent The proof of Theorem 1.1 is largely divided into sections 3 and 4. Each section has various subsections named 
according to the various sub-cases mentioned in the statement of the theorem.


\section{\bf{$W\neq\emptyset$; \,\bf{$r=1,r'=2$} }}

In this case, because of $r < r'$ we have $\nu = \lambda + \mu$ and $q'< q$. Moreover, it is easy 
to see that $\mu = 0$ and $q-q'= 1$ can not happen simultaneously. Let, 
$\mathcal{G} = \{\xi_{11}, \phi_{0}, \phi_{1}, \psi_{0}, \theta\}$, such that 
\begin{eqnarray*}
\xi_{11} & = & X_{1}^2-X_{0}X_{2}\\
\phi_{0} & = & X_1X_{2}^q-X_{0}^{\lambda}Y^{w}\\
\phi_{1} & = & X_{2}^{q+1}-X_{0}^{\lambda -1}X_1Y^{w}\\
\psi_{0} & = & X_2^{q'+1}Y^{v-w}-X_{0}^{\nu}= X_2^{q'+1}Y^{v-w}-X_{0}^{\lambda+\mu}\\
\theta & = & Y^{v}-X_{0}^{\mu}X_{1}X_{2}^{q-q'-1}\\
\end{eqnarray*}

\noindent Let $A=[\xi_{11}, \phi_{0},\phi_{1},\psi_{0},\theta ]$, which is the 0-th syzygy matrix. 
We know that $\mathcal{G}$ is a Gr\"{o}bner Basis with respect to the graded reverse lexicographic 
order from the work done in \cite{seng1}. We indicate below the exact results of \cite{seng1}, which 
have been used together with Theorem \ref{schreyer} for computing the generators for the first syzygy module. 

\begin{eqnarray*}
R_{1} & = & \left(-X_{2}^{q}, X_{1}, -X_{0}, 0, 0 \right),\quad {\rm by \quad [6.1, \mbox{\cite{seng1}}]}\\
R_{2} & = & \left(-X_{2}^{q+1} + X_{0}^{\lambda-1}X_{1}Y^w, 0, X_{1}^2 - X_{0}X_{2}, 0, 0 \right),\quad 
{\rm by \quad [3.1, \mbox{\cite{seng1}}]}\\
R_{3} & = & \left(X_{0}^{\lambda + \mu} - X_{2}^{q'+1}Y^{v-w}, 0, 0, X_{1}^2 - X_{0}X_{2}, 0 \right), \quad {\rm by \quad [3.1, \mbox{\cite{seng1}}]}\\
R_{4} & = & \left(X_{0}^{\mu}X_{1}X_{2}^{q-q'-1} - Y^v, 0, 0, 0, X_{1}^2 - X_{0}X_{2}\right), \quad {\rm by \quad [3.1, \mbox{\cite{seng1}}]}\\
R_{5} & = & \left(Y^{w} X_{0}^{\lambda-1}, -X_{2}, X_{1}, 0, 0\right), \quad {\rm by \quad [4.5, \mbox{\cite{seng1}}]}\\
R_{6} & = & \left(0, -Y^{v-w}, 0, X_{1}X_{2}^{q-q'-1}, -X_{0}^{\lambda}\right), \quad {\rm by \quad [8.4, \mbox{\cite{seng1}}]}\\
R_{7} & = & \left(0, X_{0}^{\mu}X_{1}X_{2}^{q-q'-1} - Y^{v}, 0, 0, X_{1}X_{2}^{q} - X_{0}^{\lambda}Y^w\right), 
\quad {\rm by \quad [3.1, \mbox{\cite{seng1}}]}\\
R_{8} & = & \left(-X_{0}^{\mu +\lambda-1}X_{2}^{q-q'-1}, 0, -Y^{v-w}, X_{2}^{q-q'}, -X_{0}^{\lambda-1}X_{1}\right), 
\quad {\rm by \quad [8.1, \mbox{\cite{seng1}}]}\\
R_{9} & = & \left(0, 0, X_{0}^{\mu}X_{1}X_{2}^{q-q'-1} - Y^{v}, 0, X_{2}^{q+1} - X_{0}^{\lambda-1}X_{1}Y^w\right),
\quad {\rm by \quad [3.1, \mbox{\cite{seng1}}]}\\
R_{10} & = & \left(0, X_{0}^{\mu}, 0, -Y^w, X_{2}^{q'+1}\right), \quad {\rm by \quad [9.2, \mbox{\cite{seng1}}]}
\end{eqnarray*}
\noindent We observe that 
\begin{eqnarray*}
R_{2} & = & X_{1} \cdot R_{5} + X_{2} \cdot R_{1}\\
R_{7} & = & Y^w \cdot R_{6} + (X_{1}X_{2}^{q-q'-1}) \cdot R_{10}\\
R_{9} & = & Y^w \cdot R_{8} + X_{2}^{q-q'} \cdot R_{10} + X_{0}^{\mu}X_{2}^{q-q'-1} \cdot R_{5}
\end{eqnarray*}
\noindent After removing $R_{2}$, $R_{7}$, $R_{9}$ from the list we get our 1st syzygy matrix 
\begin{eqnarray*}
B & = & \left[{}^{t}R_{1}, ^{t}R_{3}, ^{t}R_{4}, ^{t}R_{5}, ^{t}R_{6}, ^{t}R_{8}, ^{t}R_{10}\right]\\
{} & = & {\scriptsize \left[\begin{matrix}
-X_{2}^q & X_{0}^{\lambda + \mu}-X_{2}^{q'+1}Y^{v-w} & X_{0}^{\mu}X_{1}X_{2}^{q-q'-1}-Y^{v} & Y^{w}X_{0}^{\lambda-1} & 0 &-X_{0}^{\lambda+\mu-1}X_{2}^{q-q'-1} & 0\\
X_{1} & 0 & 0 &-X_{2} & -Y^{v-w} & 0 & X_{0}^{\mu}\\
-X_{0} & 0 & 0 & X_{1} & 0 & -Y^{v-w} & 0\\
0 & X_{1}^2-X_{0}X_{2} & 0 & 0 & X_{1}X_{2}^{q-q'-1} & X_{2}^{q-q'} & -Y^{w}\\
0 & 0 & X_{1}^2-X_{0}X_{2} & 0 & -X_{0}^{\lambda}&-X_{0}^{\lambda-1}X_{1} & X_{2}^{q'+1}
\end{matrix}
\right]}
\end{eqnarray*}

\noindent In order to determine the second syzygy we proceed to determine the kernel of the map given by the matrix $B$. 
Suppose that 

$$B \, \cdot \, 
\left[
\begin{matrix}
f_{1}\\
f_{2}\\
f_{3}\\
f_{4}\\
f_{5}\\
f_{6}\\
f_{7}
\end{matrix}
\right]
=0$$

\noindent Multiplying the 3rd row with the column vector we get
\begin{equation}
\label{eq1} -X_{0}f_{1} + X_{1}f_{4}-Y^{v-w}f_{6} = 0 
\end{equation}
and therefore
$$X_{1}f_{4}\in \langle -X_{0},Y^{v-w} \rangle .$$
\smallskip

\noindent Therefore $f_{4}\in \langle X_{0},Y^{v-w} \rangle $, since 
$X_{0},Y^{v-w}, X_{1}$ form a regular sequence in $R$. Hence, we can write 
$$f_{4} = X_{0}p_{1} + Y^{v-w}p_{2},$$ 
\smallskip

\noindent for some polynomials $p_{1}$ and $p_{2}$. Plugging this value in (3.1) we get 
$$X_{0}(f_{1} - X_{1}p_{1}) = Y^{v-w}(X_{1}p_{2} - f_{6})$$
\smallskip

\noindent which implies that $X_{0}\mid X_{1}p_{2} - f_{6}$ and $Y^{v-w}\mid f_{1} - X_{1}p_{1}$ (follows from fact that R is a UFD). 
If we write $f_{1} = X_{1}p_{1}+Y^{v-w}p_{3}$, we finally obtain 
\begin{eqnarray*}
f_{1} & = & X_{1}p_{1} + Y^{v-w}p_{3}\\
f_{6} & = & -X_{0}p_{3} + X_{1}p_{2}\\
f_{4} & = & X_{0}p_{1} + Y^{v-w}p_{2}.
\end{eqnarray*}

\noindent Multiplying the 2nd row of $A$ with the column vector we get
\begin{equation}
\label{eq2} X_{1}f_{1}-X_{2}f_{4}-Y^{v-w}f_{5}+X_{0}^{\mu}f_{7}=0
\end{equation}
\smallskip

\noindent Eliminating $f_{4}$ from (3.1) and (3.2) we obtain
\begin{equation}
\label{eq3} (X_{1}^2-X_{0}X_{2})f_{1}=X_{1}Y^{v-w}f_{5}-X_{0}^{\mu}X_{1}f_{7}+X_{2}Y^{v-w}f_{6}
\end{equation}
\smallskip

\noindent and eliminating $f_{1}$ from (3.1) and (3.2) we obtain
\begin{equation}
\label{eq4} (X_{1}^2-X_{0}X_{2})f_{4}=X_{0}Y^{v-w}f_{5}-X_{0}^{\mu+1}f_{7}+X_{1}Y^{v-w}f_{6}
\end{equation}
\smallskip

\noindent Eliminating $f_{6}$ from (3.3) and (3.4) we obtain
$$Y^{v-w}f_{5}-X_{0}^{\mu}f_{7} = X_{1}f_{1} - X_{2}f_{4}.$$
\smallskip

\noindent Plugging in the values of previously obtained $f_{1}$ and $f_{4}$ in the above expression we get 
$$Y^{v-w}(f_{5} + X_{2}p_{2} - X_{1}p_{3}) = X_{0}^{\mu}f_{7} + (X_{1}^2 - X_{0}X_{2})p_{1}.$$ 
\smallskip

\noindent Taking either side of the equality as $Y^{v-w}p_{4}$ we get
$$X_{0}^{\mu}f_{7}+(X_{1}^2 - X_{0}X_{2})p_{1} = Y^{v-w}p_{4}$$

$$f_{5} = p_{4} + X_{1}p_{3} - X_{2}p_{2}.$$
\smallskip

\noindent We now multiply the fourth and the fifth rows of the matrix $A$ with the column vector to get 
\begin{equation}
(X_{1}^2-X_{0}X_{2})f_{2} = -X_{1}X_{2}^{q-q'-1}f_{5}-X_{2}^{q-q'}f_{6}+Y^{w}f_{7}
\end{equation}
\smallskip

\begin{equation}
(X_{1}^2-X_{0}X_{2})f_{3}=X_{0}^{\lambda}f_{5}+X_{0}^{\lambda-1}X_{1}f_{6}-X_{2}^{q'+1}f_{7}
\end{equation}
\smallskip

\noindent Plugging these values in (3.6) we get  
$$(X_{1}^2 - X_{0}X_{2})(X_{0}^{\mu}f_{3} - X_{2}^{q'+1}p_{1} - X_{0}^{\nu-1}p_{2})=p_{4}(X_{0}^{\nu} - X_{2}^{q'+1}y^{v-w});$$ 
\smallskip

\noindent which implies that $(X_{1}^2 - X_{0}X_{2})\mid p_{4}$. Taking $p_{4}=(x_{1}^2-x_{0}x_{2})p_{5}$, 
we get 
$$X_{0}^{\mu}(f_{3} - X_{0}^{\lambda-1}p_{2} - X_{0}^{\lambda}p_{5}) = X_{2}^{q'+1}(p_{1} - Y^{v-w}p_{5}).$$ 
\smallskip

\noindent Therefore, $X_{0}^{\mu}\mid(p_{1} - Y^{v-w}p_{5})$ and 
$X_{2}^{q'+1}\mid(f_{3} - X_{0}^{\lambda-1}p_{2} - X_{0}^{\lambda}p_{5})$. Taking either quotients as $p_{6}$ we get 
\begin{eqnarray*}
p_{1} & = & Y^{v-w}p_{5} + X_{0}^{\mu}p_{6}\\[2mm]
f_{3} & = & X_{2}^{q'+1}p_{6} + X_{0}^{\lambda-1}p_{2} + X_{0}^{\lambda}p_{5}\\[2mm]
f_{7} & = & -(X_{1}^2 - X_{0}X_{2})p_{6}
\end{eqnarray*}
\smallskip

\noindent Next, plugging these values in (3.5), we get 
$$f_{2}= -X_{2}^{q-q'-1}p_{3} - X_{1}X_{2}^{q-q'-1}p_{5} - Y^{w}p_{6}.$$
Therefore, we obtain 
$$p_{1} = Y^{v-w}p_{5} + X_{0}^{\mu}p_{6}$$
\smallskip

\noindent and
\begin{eqnarray*}
f_{1} & = & X_{1}p_{1} + Y^{v-w}p_{3}\\[2mm]
f_{2} & = & -X_{2}^{q-q'-1}p_{3} - X_{1}X_{2}^{q-q'-1}p_{5} - Y^{w}p_{6}\\[2mm]
f_{3} & = & X_{0}^{\lambda-1}p_{2} + X_{0}^{\lambda}p_{5} + X_{2}^{q'+1}p_{6}\\[2mm]
f_{4} & = & X_{0}p_{1} + Y^{v-w}p_{2}\\[2mm]
f_{5} & = & (X_{1}^2 - X_{0}X_{2})p_{5} + X_{1}p_{3} - X_{2}p_{2}\\[2mm]
f_{6} & = & -X_{0}p_{3} + X_{1}p_{2}\\[2mm]
f_{7} & = & -(X_{1}^{2} - X_{0}X_{2})p_{6}
\end{eqnarray*}
\smallskip

\noindent The tuple $(f_{1}, \ldots , f_{7})$ indeed belongs to the 
kernel of the linear map defined by B, for any choice of $p_{i}$'s. 
Moreover,
$${\scriptsize \left[
\begin{matrix}
0 & Y^{v-w} & X_{1}Y^{v-w} & X_{1}X_{0}^{\mu}\\
0 & -X_{2}^{q-q'-1} & -X_{1}X_{2}^{q-q'-1} & -Y^{w}\\
X_{0}^{\lambda-1} & 0 & X_{0}^{\lambda} & X_{2}^{q'+1}\\
Y^{v-w} & 0 & X_{0}Y^{v-w} & X_{0}^{\mu+1}\\
-X_{2} & X_{1} & X_{1}^{2} - X_{0}X_{2} & 0\\
X_{1} & -X_{0} & 0 & 0\\
0 & 0 & 0 & -(X_{1}^2 - X_{0}X_{2})
\end{matrix}
\right]} \cdot 
{\scriptsize \left[
\begin{matrix}
p_{2}\\[2mm]
p_{3}\\[2mm]
p_{5}\\[2mm]
p_{6}
\end{matrix}
\right]} = 
{\scriptsize \left[
\begin{matrix}
f_{1}\\f_{2}\\f_{3}\\f_{4}\\f_{5}\\f_{6}\\f_{7}
\end{matrix}
\right]}$$
\smallskip

\noindent Hence the matrix 
$${\scriptsize \left[\begin{matrix}
0 & Y^{v-w} & X_{1}Y^{v-w} & X_{1}X_{0}^{\mu}\\
0 & -X_{2}^{q-q'-1} & -X_{1}X_{2}^{q-q'-1} & -Y^{w}\\
X_{0}^{\lambda-1} & 0 & X_{0}^{\lambda} & X_{2}^{q'+1}\\
Y^{v-w} & 0 & X_{0}Y^{v-w} & X_{0}^{\mu+1}\\
-X_{2} & X_{1} & (X_{1}^{2} - X_{0}X_{2}) & 0\\
X_{1} & -X_{0} & 0 & 0\\
0 & 0 & 0 & -(X_{1}^2 - X_{0}X_{2})
\end{matrix}
\right]}$$
\smallskip

\noindent is the second syzygy matrix. Let $L_{i}$ denote the $i$-th column from the left 
of the above matrix. We observe that 
$L_{3}=x_{0}\cdot L_{1}+x_{1}\cdot L_{2}$. If we remove the third column of the matrix, 
the second syzygy matrix in its reduced form is
$${\scriptsize C = \left[
\begin{matrix}
0 & Y^{v-w} & X_{1}X_{0}^{\mu}\\
0 & -X_{2}^{q-q'-1} & -Y^{w}\\
X_{0}^{\lambda-1} & 0 & X_{2}^{q'+1}\\
Y^{v-w} & 0 & X_{0}^{\mu+1}\\
-X_{2} & X_{1} & 0\\
X_{1} & -X_{0} & 0\\
0 & 0 & -(X_{1}^2 - X_{0}X_{2})
\end{matrix}
\right]}$$
\smallskip

Now we calculate the third syzygy, that is the kernel of the map given by the matrix C 
under choice of standard basis. Let $(g_{1},g_{2},g_{3})$ be an element of the kernel. This means 
that 
$${\scriptsize\left[
\begin{matrix}
0 & Y^{v-w} & X_{1}X_{0}^{\mu}\\
0 & -X_{2}^{q-q'-1} & -Y^{w}\\
X_{0}^{\lambda-1} & 0 & X_{2}^{q'+1}\\
Y^{v-w} & 0 & X_{0}^{\mu+1}\\
-X_{2} & X_{1} & 0\\
X_{1} & -X_{0} & 0\\
0 & 0 & -(X_{1}^2 - X_{0}X_{2})
\end{matrix}
\right]} \cdot 
{\scriptsize \left[
\begin{matrix}
g_{1}\\[2mm]
g_{2}\\[2mm]
g_{3}
\end{matrix}
\right]} \quad = \quad 0.$$
\smallskip

\noindent Multiplication of the 7th row with the column vector gives 
$-(X_{1}^2 - X_{0}X_{2})g_{3} = 0$ which implies $g_{3}=0$. 
Multiplication of the 5th and 6th row with the column vector gives 
$$X_{1}g_{1} - X_{2}g_{2} = 0$$ and 
$$-X_{2}g_{1} + X_{1}g_{2} = 0$$ implying that $g_{1} = g_{2} = 0$. 
So we get that the map given by matrix C is injective. 
\medskip

It follows from our preceding discussion that 
$$0\longrightarrow R^{3}\stackrel{\theta{_{C}}}{\longrightarrow} R^{7}\stackrel{\theta{_{B}}}{\longrightarrow} R^{5}\stackrel{\theta{_{A}}}{\longrightarrow} R\longrightarrow k[\Gamma]\longrightarrow0$$ 
is a free resolution of the monomial curve in question, where $\theta_{A},\theta_{B},\theta_{C}$ are maps given by the 
matrices $A$, $B$ and $C$ respectively. What we are yet to achieve is its minimality. The resolution would be minimal 
if and only if each and every entry of the matrices $A$, $B$ and $C$ belong to the maximal ideal 
$(X_{0}, X_{1}, X_{2}, Y)$. A scrutiny of the entries show that the free resolution obtained above 
is minimal if and only if $\mu\neq 0,q-q'\neq 1,\lambda \neq 1$. Therefore, following are the cases 
in which we have non-minimality:

\begin{enumerate}
\item[(a)] $\mu=0$, $q-q'\neq 1$;
\smallskip

\item[(b)] $\mu\neq 0$, $q-q'=1$, $\lambda\neq 1$;
\smallskip

\item[(c)] $\mu\neq 0$, $q-q'\neq 1$, $\lambda=1$;
\smallskip

\item[(d)] $\mu\neq 0,q-q'=1$, $\lambda=1$.
\end{enumerate}

\noindent The only  case of possible non-minimality that we have not considered is when $\mu=0$, $q-q'=1$, for the 
reason that it can not occur under the given conditions. Therefore,  In order to extract a minimal free resolution 
from the non-minimal one we apply the elementary matrices $E_{ij}(*)$ and invoke Lemma \ref{homo1} and Corollary \ref{homo2}. The matrices $E_{ij}(\alpha)$ would denote the matrix whose $(i,j)^{th}$ entry is $\alpha$, all diagonal entires 1 and rest of the entries are 0. 
\medskip

\noindent\textbf{Case (a): $\mu=0$, \, $q-q'\neq 1$}
\smallskip

\noindent In this case, we see that the $(2,7)$-th entry of matrix $B$ is 1.  
Let $P_{3} = E_{42}(Y^{w})E_{52}(-X_{2}^{(q'+1)})$ and 
$P_{2} = E_{71}(-X_{1})E_{74}(X_{2})E_{75}(Y^{v-w})$. Then
$${\tiny P_{3}BP_{2} = \left[
\begin{matrix}
-X_{2}^{q} & X_{0}^{\lambda} - X_{2}^{q'+1}Y^{v-w}& X_{1}X_{2}^{q-q'-1}-Y^{v} & Y^{w} X_{0}^{\lambda-1} & 0 & -X_{0}^{\lambda-1}X_{2}^{q-q'-1} & 0\\
0 & 0 & 0 & 0 & 0 & 0 & 1 \\
-X_{0} & 0 & 0 & X_{1} & 0 & -Y^{v-w} & 0\\
X_{1}Y^{w} & X_{1}^2-X_{0}X_{2} & 0 & -X_{2}Y^{w} & X_{1}X_{2}^{q-q'-1}-Y^{v} & X_{2}^{q-q'} & 0\\
-X_{1}X_{2}^{q'+1} & 0 & X_{1}^2 - X_{0}X_{2} & X_{2}^{q'+2} & Y^{v-w}X_{2}^{q'+1} - X_{0}^{\lambda} & -X_{0}^{\lambda-1}X_{1} & 0
\end{matrix}
\right]},$$

$${\scriptsize P_{2}^{-1}C = \left[
\begin{matrix}
0 & Y^{v-w} & X_{1}\\
0 & -X_{2}^{q-q'-1} & -Y^{w}\\
X_{0}^{\lambda-1} & 0 & X_{2}^{q'+1}\\
Y^{v-w} & 0 & X_{0}\\
-X_{2} & X_{1} & 0\\
X_{1} & -X_{0} & 0\\
0 & 0 & 0
\end{matrix}
\right]} \quad {\rm and} \quad 
AP_{3}^{-1} = [\xi_{11},0,\phi_{1},\psi_{0},\theta].$$

\noindent Now we apply Corollary \ref{homo2}. and obtain the minimal free resolution 

$$0\longrightarrow R^{3}\stackrel{\theta_{C}}{\longrightarrow} R^{6}\stackrel{\theta_{B}}{\longrightarrow} R^{4}\stackrel{\theta_{A}}{\longrightarrow} R\longrightarrow k[\Gamma]\longrightarrow 0,$$ 

\noindent where the maps ${\theta_{A}}$, ${\theta_{B}}$ and ${\theta_{C}}$ are given by the syzygy matrices $A$, $B$ and $C$ respectively, given by
$$A=[\xi_{11},\phi_{1},\psi_{0},\theta],$$

$${\scriptsize B=\left[
\begin{matrix}
-X_{2}^{q} & X_{0}^{\lambda } - X_{2}^{q'+1}Y^{v-w}& X_{1}X_{2}^{q-q'-1}-Y^{v} & Y^{w} X_{0}^{\lambda-1} & 0 & -X_{0}^{\lambda-1}X_{2}^{q-q'-1}\\
-X_{0} &0& 0 & X_{1} & 0 & -Y^{v-w}\\
X_{1}Y^{w} &X_{1}^2-X_{0}X_{2}& 0 & -X_{2}Y^{w} & X_{1}X_{2}^{q-q'-1}-Y^{v} & X_{2}^{q-q'}\\
-X_{1}X_{2}^{q'+1} &0& X_{1}^2 - X_{0}X_{2} & X_{2}^{q'+2} & Y^{v-w}X_{2}^{q'+1} - X_{0}^{\lambda} & -X_{0}^{\lambda-1}X_{1}
\end{matrix}
\right]},$$

$${\scriptsize C = \left[
\begin{matrix}
0 & Y^{v-w} & X_{1}\\
0 & -X_{2}^{q-q'-1} & -Y^{w}\\
X_{0}^{\lambda-1} & 0 & X_{2}^{q'+1}\\
Y^{v-w} & 0 & X_{0}\\
-X_{2} & X_{1} & 0\\
X_{1} & -X_{0} & 0
\end{matrix}
\right]}.$$
The entries of the matrices are from the maximal ideal and hence the resolution 
is minimal. The total Betti numbers in this case are [4,6,3].
\bigskip

\noindent\textbf{Case (b): $\mu\neq 0$, $\lambda\neq 1$, $q-q'=1$}
\smallskip

\noindent We take $P_{2}=E_{12}(-Y^{v-w})E_{52}(-X_{1})E_{62}(X_{0})$ and $P_{1}=E_{23}(-Y^{w})$ and proceeding as before we obtain 
the minimal free resolution 

$$0\longrightarrow R^{2}\stackrel{\theta_{C}}{\longrightarrow} R^{6}\stackrel{\theta_{B}}{\longrightarrow} R^{5}\stackrel{\theta_{A}}{\longrightarrow} R\longrightarrow k[\Gamma]\longrightarrow 0,$$ 

\noindent with the syzygy matrices given by 

$$A = [\xi_{11},\phi_{0},\phi_{1},\psi_{0},\theta],$$

$${\scriptsize B=\left[
\begin{matrix}
-X_{2}^{q} & X_{0}^{\mu}X_{1}-Y^{v} & Y^{w}X_{0}^{\lambda-1} & 0 & -X_{0}^{\lambda+\mu-1} & 0\\
X_{1} & 0 & -X_{2} & -Y^{v-w} & 0 & X_{0}^{\mu}\\
-X_{0} & 0 & X_{1} & 0 & -Y^{v-w} & 0\\
0 & 0 & 0 & X_{1} & X_{2} & -Y^{w}\\
0 & X_{1}^2-X_{0}X_{2} & 0 & -X_{0}^{\lambda} & -X_{0}^{\lambda-1}X_{1} & X_{2}^{q'+1}
\end{matrix}
\right]},$$

$${\scriptsize C=\left[
\begin{matrix}
0 & X_{1}X_{0}^{\mu}- Y^{v}\\
X_{0}^{\lambda-1} & X_{2}^{q'+1}\\
Y^{v-w} & X_{0}^{\mu+1}\\
-X_{2} & -X_{1}Y^{w}\\
X_{1} & X_{0}Y^{w}\\
0 & -(X_{1}^{2} - X_{0}X_{2})
\end{matrix}
\right]}.$$
Therefore, the total Betti numbers are $[5,6,2]$.
\bigskip

\noindent\textbf{Case (c): $\mu\neq 0$, $\lambda = 1$, $q-q'\neq 1$}
\smallskip

\noindent This will affect the $(3,1)$-th entry of C. We take $P_{2}=E_{43}(Y^{v-w})E_{53}(-X_{2})E_{63}(X_{1})$ 
and $P_{1}=E_{13}(-X_{2}^{q'+1})$. So, again applying the lemmas we get the minimal free resolution as 

$$0\longrightarrow R^{2}\stackrel{\theta_{C}}{\longrightarrow} R^{6}\stackrel{\theta_{B}}{\longrightarrow} R^{5}\stackrel{\theta_{A}}{\longrightarrow} R\longrightarrow k[\Gamma]\longrightarrow 0,$$ 

\noindent with the syzygy matrices given by 
 
$$A = [\xi_{11},\phi_{0},\phi_{1},\psi_{0},\theta],$$
$${\scriptsize B = \left[
\begin{matrix}
-X_{2}^{q} & X_{0}^{\mu +1} - X_{2}^{q'+1}Y^{v-w} & Y^{w} & 0 & -X_{0}^{\mu}X_{2}^{q-q'-1} & 0\\
X_{1} & 0 & -X_{2} & -Y^{v-w} & 0 & X_{0}^{\mu}\\
-X_{0} & 0 & X_{1} & 0 & -Y^{v-w} & 0\\
0 & X_{1}^{2} - X_{0}X_{2} & 0 & X_{1}X_{2}^{q-q'-1} & X_{2}^{q-q'} & -Y^{w}\\
0 & 0 & 0 & -X_{0} & -X_{1} & X_{2}^{q'+1}
\end{matrix}
\right]},$$
$${\scriptsize C=\left[
\begin{matrix}
Y^{v-w} & X_{1}X_{0}^{\mu}\\
-X_{2}^{q-q'-1} & -Y^{w}\\
0 & X_{0}^{\mu+1}-X_{2}^{q'+1}Y^{v-w}\\
X_{1} & X_{2}^{q'+2}\\
-X_{0} & -X_{2}^{q'+1}X_{1}\\
0 & -(X_{1}^{2} - X_{0}X_{2})
\end{matrix}\right]}.$$
Therefore, the total Betti numbers are $[5,6,2]$.
\bigskip

\noindent\textbf{Case (d): $\mu\neq 0$, $\lambda = 1$, $q-q' = 1$}
\smallskip

\noindent Here, along with the preceding case, the $(2,1)$-th entry of the 
second syzygy matrix in preceding case will be affected. So we take 
$P_{2}=E_{12}(-Y^{v-w})E_{42}(-X_{1})E_{52}(X_{0})$ and 
$P_{1}=E_{12}(-Y^{w})$. The minimal free resolution which we obtain 
is 

$$0\longrightarrow R^{1}\stackrel{\theta_{C}}{\longrightarrow} R^{5}\stackrel{\theta_{B}}{\longrightarrow} R^{5}\stackrel{\theta_{A}}{\longrightarrow} R\longrightarrow k[\Gamma]\longrightarrow 0,$$ 

\noindent with the syzygy matrices given by 
$$A = [\xi_{11},\phi_{0},\phi_{1},\psi_{0},\theta],$$
$${\scriptsize B=\left[
\begin{matrix}
-X_{2}^{q} & Y^{w} & 0 & -X_{0}^{\mu} & 0\\
X_{1} & -X_{2} & -Y^{v-w} & 0 & X_{0}^{\mu}\\
-X_{0} & X_{1} & 0 & -Y^{v-w} & 0\\
0 & 0 & X_{1} & X_{2} & -Y^{w}\\
0 & 0 & -X_{0} & -X_{1} & X_{2}^{q'+1}
\end{matrix}
\right]}$$
$${\scriptsize C=\left[
\begin{matrix}
X_{1}X_{0}^{\mu}-Y^{v}\\
X_{2}^{q'+1}\\
-X_{0}^{\mu+1}-X_{2}^{q'+1}Y^{v-w}\\
X_{2}^{q'+2}-X_{1}Y^{w}\\
X_{0}Y^{w}-X_{2}^{q'+1}X_{1}\\
-(X_{1}^{2} - X_{0}X_{2})
\end{matrix}\right]}.$$
Therefore, the total Betti numbers are $[5,5,1]$.
 
\section{\bf{$W\neq\emptyset ;\, r=1,r'=1$}}
\noindent In this case, because of $r = r'$ we have $\nu = \lambda + \mu$ and $q' < q$. Let 
$\mathcal{G} = \{\xi_{11}, \phi_{0}, \phi_{1}, \psi_{0}, \psi_{1}, \theta\}$, such that 
\begin{eqnarray*}
\xi_{11} & = & X_{1}^2-X_{0}X_{2}\\
\phi_{0} & = & X_{1}X_{2}^{q} - X_{0}^{\lambda}Y^{w}\\
\phi_{1} & = & X_{2}^{q+1} - X_{0}^{\lambda -1}X_{1}Y^{w}\\
\psi_{0} & = & X_{1}X_{2}^{q'}Y^{v-w} - X_{0}^{\lambda + \mu}\\
\psi_{1} & = & X_{2}^{q'+1}Y^{v-w} - X_{0}^{\lambda + \mu - 1}X_{1}\\
\theta & = & Y^{v} - X_{2}^{q-q'}X_{0}^{\mu}\\
\end{eqnarray*}

\noindent Let $A=[\xi_{11}, \phi_{0},\phi_{1},\psi_{0},\psi_{1}, \theta ]$, which is the 0th syzygy matrix. We know that 
$\mathcal{G}$ is a Gr\"{o}bner Basis with respect to the graded reverse lexicographic order from the work done in 
\cite{seng1}. We indicate below the exact results of \cite{seng1}, which have been used together with Theorem \ref{schreyer} 
for computing the generators for the first syzygy module. 

\begin{eqnarray*}
R_{1} & = & \left(-X_{2}^{q}, X_{1}, -X_{0}, 0, 0, 0\right) \quad {\rm by \quad [6.1, \mbox{\cite{seng1}}]}\\
R_{2} & = & \left(X_{0}^{\lambda-1}X_{1}Y^{w} - X_{2}^{q+1}, 0, X_{1}^{2} - X_{0}X_{2}, 0, 0, 0\right)
\quad {\rm by \quad [3.1, \mbox{\cite{seng1}}]}\\
R_{3} & = & \left(-X_{2}^{q'}Y^{v-w}, 0, 0, X_{1}, -X_{0}, 0\right)
\quad {\rm by \quad [7.2, \mbox{\cite{seng1}}]}\\
R_{4} & = & \left(X_{0}^{\lambda + \mu-1}X_{1} - X_{2}^{q'+1}Y^{v-w}, 0, 0, 0, X_{1}^{2} - X_{0}X_{2}, 0\right)
\quad {\rm by \quad [3.1, \mbox{\cite{seng1}}]}\\
R_{5} & = & \left(X_{0}^{\mu}X_{2}^{q-q'}-Y^{v}, 0, 0, 0, 0, X_{1}^2-X_{0}X_{2}\right)
\quad {\rm by \quad [3.1, \mbox{\cite{seng1}}]}\\
R_{6} & = & \left(X_{0}^{\lambda-1}Y^{w}, -X_{2}, X_{1}, 0, 0, 0\right)
\quad {\rm by \quad [4.5, \mbox{\cite{seng1}}]}\\
R_{7} & = & \left(0, -Y^{v-w}, 0, X_{2}^{q-q'}, 0, -X_{0}^{\lambda}\right)
\quad {\rm by \quad [8.1, \mbox{\cite{seng1}}]}\\
R_{8} & = & \left(X_{0}^{\lambda + \mu-1}X_{2}^{q-q'-1}, -Y^{v-w}, 0, 0, X_{1}X_{2}^{q-q'-1}, -X_{0}^{\lambda}\right)
\quad {\rm by \quad [8.5, \mbox{\cite{seng1}}]}\\
R_{9} & = & \left(0, X_{0}^{\mu}X_{2}^{q-q'}-Y^{v}, 0, 0, 0, X_{1}X_{2}^{q} - X_{0}^{\lambda}Y^{w}\right)
\quad {\rm by \quad [3.1, \mbox{\cite{seng1}}]}\\
R_{10} & = & \left(-X_{0}^{\lambda + \mu-1}X_{2}^{q-q'}, 0, -X_{1}Y^{v-w}, X_{2}^{q-q'+1}, 0, -X_{0}^{\lambda-1}X_{1}^{2}\right)
\quad {\rm by \quad [8.6, \mbox{\cite{seng1}}]}\\
R_{11} & = & \left(0, 0, -Y^{v-w}, 0, X_{2}^{q-q'}, -X_{0}^{\lambda-1}X_{1}\right)
\quad {\rm by \quad [8.1, \mbox{\cite{seng1}}]}\\
R_{12} & = & \left(0, 0, X_{0}^{\mu}X_{2}^{q-q'} - Y^{v}, 0, 0, -X_{0}^{\lambda -1}X_{1}Y^{w} + X_{2}^{q+1}\right)
\quad {\rm by \quad [3.1, \mbox{\cite{seng1}}]}\\
R_{13} & = & \left(X_{0}^{\lambda + \mu-1}, 0, 0, -X_{2}, X_{1}, 0\right)\quad {\rm by \quad [4.5, \mbox{\cite{seng1}}]}\\
R_{14} & = & \left(0, X_{0}^{\mu}, 0, -Y^{w}, 0, X_{1}X_{2}^{q'}\right)\quad {\rm by \quad [9.2, \mbox{\cite{seng1}}]}\\
R_{15} & = & \left(0, 0, X_{0}^{\mu}, 0, -Y^{w}, X_{2}^{q'+1}\right)\quad {\rm by \quad [9.2, \mbox{\cite{seng1}}]}\\
\end{eqnarray*}
\noindent We observe that 
\begin{eqnarray*}
R_{2} & = & X_{2} \cdot R_{1} + X_{1} \cdot R_{6}\\
R_{4} & = & X_{1} \cdot R_{13} + X_{2} \cdot R_{3}\\
R_{8} & = & X_{2}^{q-q'-1} \cdot R_{13} + R_{7}\\
R_{9} & = & Y^{w}\cdot R_{7} + X_{2}^{q-q'}\cdot R_{14}\\
R_{10} & = & -X_{2}^{q-q'}\cdot R_{13} + X_{1}\cdot R_{11}\\
R_{12} & = & Y^{w}\cdot R_{11} + X_{2}^{q-q'}\cdot R_{15}
\end{eqnarray*}

\noindent After removing $R_{2}$, $R_{4}$, $R_{8}$, $R_{9}$, $R_{10}$, $R_{12}$ from the list we get our 1st syzygy matrix 
\begin{eqnarray*}
B & = & \left[^{t}R_{1}, ^{t}R_{3}, ^{t}R_{5}, ^{t}R_{6}, ^{t}R_{7}, ^{t}R_{11}, ^{t}R_{13}, ^{t}R_{14}, ^{t}R_{15}\right]\\
{} & = & {\scriptsize \left [\begin{matrix}
-X_{2}^{q} & -X_{2}^{q'}Y^{v-w} & X_{0}^{\mu}X_{2}^{q-q'}-Y^{v} & Y^{w}X_{0}^{\lambda -1} & 0 & 0 & X_{0}^{\lambda+\mu -1} & 0 & 0 \\
X_{1} & 0 & 0 & -X_{2} & -Y^{v-w} & 0 & 0 & X_{0}^{\mu} & 0 \\
-X_{0} & 0 & 0 & X_{1} & 0 & -Y^{v-w} & 0 & 0 & X_{0}^{\mu} \\
0 & X_{1} & 0 & 0 & X_{2}^{q-q'} & 0 & -X_{2} & -Y^{w} & 0 \\
0 & -X_{0} & 0 & 0 & 0 & X_{2}^{q-q'} & X_{1} & 0 & -Y^{w}\\
0 & 0 & X_{1}^{2}-X_{0}X_{2} & 0 & -X_{0}^{\lambda} & -X_{1}X_{0}^{\lambda -1} & 0 & X_{1}X_{2}^{q'} & X_{2}^{q'+1}
\end{matrix}
\right]}
\end{eqnarray*}

\noindent We now determine the second  syzygy, which is the kernel of the map given by the matrix $B$. 
Let $[f_{1},f_{2},f_{3},f_{4},f_{5},f_{6},f_{7},f_{8},f_{9}]\in\mathbb{R}^{9}$ denote an element 
in the second syzygy, that is, 
$${\scriptsize \left [\begin{matrix}
-X_{2}^{q} & -X_{2}^{q'}Y^{v-w} & X_{0}^{\mu}X_{2}^{q-q'}-Y^{v} & Y^{w}X_{0}^{\lambda -1} & 0 & 0 & X_{0}^{\nu -1} & 0 & 0 \\
X_{1} & 0 & 0 & -X_{2} & -Y^{v-w} & 0 & 0 & X_{0}^{\mu} & 0 \\
-X_{0} & 0 & 0 & X_{1} & 0 & -Y^{v-w} & 0 & 0 & X_{0}^{\mu} \\
0 & X_{1} & 0 & 0 & X_{2}^{q-q'} & 0 & -X_{2} & -Y^{w} & 0 \\
0 & -X_{0} & 0 & 0 & 0 & X_{2}^{q-q'} & X_{1} & 0 & -Y^{w}\\
0 & 0 & X_{1}^{2}-X_{0}X_{2} & 0 & -X_{0}^{\lambda} & -X_{1}X_{0}^{\lambda -1} & 0 & X_{1}X_{2}^{q'} & X_{2}^{q'+1}
\end{matrix}\right]\left[\begin{matrix}
f_{1}\\
f_{2}\\
f_{3}\\
f_{4}\\
f_{5}\\
f_{6}\\
f_{7}\\
f_{8}\\
f_{9}\\
\end{matrix}\right] = 0}$$
Multiplying the column vector with 2nd and 3rd row we get 
\begin{equation}\label{eq:1}
X_{1}f_{1} - X_{2}f_{4} - Y^{v-w}f_{5} + X_{0}^{\mu}f_{8} =  0
\end{equation}

\begin{equation}\label{eq:2} 
X_{0}f_{1} + X_{1}f_{4} - Y^{v-w}f_{6} + X_{0}^{\mu}f_{9} = 0
\end{equation}
\medskip

\noindent Eliminating $f_{1}$ from the above equations we obtain 
\begin{equation}\label{eq:3}
(X_{1}^{2}-X_{0}X_{2})f_{4}=Y^{v-w}(X_{0}f_{5}+X_{1}f_{6})-X_{0}^{\mu}(X_{1}f_{9}+X_{0}f_{8})
\end{equation}
\medskip

\noindent which implies that $(X_{1}^{2}-X_{0}X_{2})f_{4}\in \langle Y^{v-w}, X_{0}^{\mu} \rangle$. 
Therefore, $f_{4}\in \langle Y^{v-w}, X_{0}^{\mu}\rangle$, since 
$Y^{v-w}, X_{0}^{\mu}, X_{1}^{2} - X_{0}X_{2}$ form a regular sequence. Hence, 

$$f_{4}= Y^{v-w}p_{1} + X_{0}^{\mu}p_{2}$$
\medskip

\noindent where $p_{1}$ and $p_{2}$ are polynomials in $R$. Now, eliminating $f_{4}$ from $(4.1)$ and $(4.2)$ we get

\begin{equation}\label{eq:4}
(X_{1}^{2}-X_{0}X_{2})f_{1} = Y^{v-w}(X_{1}f_{5} + X_{2}f_{6}) - X_{0}^{\mu}(X_{1}f_{8} + X_{2}f_{9})
\end{equation}
\medskip

\noindent and a similar argument as above shows that $f_{1}\in \langle Y^{v-w}, X_{0}^{\mu}\rangle$. 
Therefore, we may write

$$f_{1}=Y^{v-w}p_{3}+X_{0}^{\mu}p_{4}.$$

\noindent Putting $f_{4}$ in $(4.3)$ and $f_{1}$ in $(4.4)$ we get

\begin{equation}\label{eq:5}
X_{0}f_{5} + X_{1}f_{6} = (X_{1}^{2} - X_{0}X_{2})p_{1} + X_{0}^{\mu}p_{5}
\end{equation}
\begin{equation}\label{eq:6}
X_{1}f_{5} + X_{2}f_{6} = (X_{1}^{2} - X_{0}X_{2})p_{3} + X_{0}^{\mu}p_{6}
\end{equation}
\begin{equation}\label{eq:7}
X_{0}f_{8} + X_{0}f_{9} = -(X_{1}^{2} - X_{0}X_{2})p_{2} + y^{v-w}p_{5}
\end{equation}
\begin{equation}\label{eq:8}
X_{1}f_{8} + X_{2}f_{9} = -(X_{1}^{2} - X_{0}X_{2})p_{4} + Y^{v-w}p_{6}
\end{equation}    
\medskip

\noindent where the $p_{i}$'s are polynomials in $R$. Eliminating $f_{6}$ from $(4.5)$ and 
$(4.6)$ we get, 
$$(X_{1}^2 - X_{0}X_{2})(f_{5} + X_{2}p_{1} - X_{1}p_{3}) = X_{0}^{\mu}(X_{1}p_{6} - X_{2}p_{5}).$$
Proceeding as before we get
$$f_{5} = X_{1}p_{3} - X_{2}p_{1} + X_{0}^{\mu}p_{7}$$
and 
\begin{equation}\label{eq:9}
(X_{1}^{2} - X_{0}X_{2})p_{7} = -p_{5}X_{2} + p_{6}X_{1}
\end{equation}

\noindent Similarly, eliminating $f_{5}$ from $(4.5)$ and $(4.6)$ we get, 
$$(X_{1}^2 - X_{0}X_{2})(f_{6} - X_{1}p_{1} + X_{0}p_{3}) = X_{0}^{\mu}(-X_{0}p_{6} + X_{1}p_{5})$$
Proceeding as before we get
$$f_{6} = X_{1}p_{1} - X_{0}p_{3} + X_{0}^{\mu}p_{8}$$

\begin{equation}\label{eq:10}
(X_{1}^{2}-X_{0}x_{2})p_{8} = -p_{5}X_{1} - p_{6}X_{0}
\end{equation}
\medskip

\noindent Similarly, from $(4.7)$ and $(4.8)$ we get,
$$f_{8}=-X_{1}p_{4} + X_{2}p_{2} + y^{v-w}p_{9}$$
$$f_{9}=-X_{1}p_{2} + X_{0}p_{4} + Y^{v-w}p_{10}$$
$$(X_{1}^{2} - X_{0}X_{2})p_{9} = -p_{5}X_{2} - p_{6}X_{1}$$
$$(X_{1}^{2} - X_{0}X_{2})p_{10} = -p_{5}X_{1} - p_{6}X_{0}$$

\noindent Hence, it follows that 
$$ p_{5}= X_{1}p_{8} + X_{0}p_{9}$$
$$p_{6}= X_{1}p_{9} + X_{2}p_{8}$$

\noindent Multiplying the fourth row and the fifth row of the matrix $B$ with the column 
vector and proceeding exactly as before we obtain
$$f_{7}=-Y^{w}p_{2}-X_{2}^{q-q'}p_{1}+ p_{11}(Y^{v}-X_{2}^{q-q'}p_{1})$$
$$f_{2}=-Y^{w}p_{4}-X_{2}^{q-q'}p_{3}+ p_{12}(Y^{v}-X_{2}^{q-q'}p_{3})$$
$$p_{5}=(X_{1}^{2}-X_{0}X_{2})p_{11}$$
$$p_{6}=(X_{1}^{2}-X_{0}X_{2})p_{12}$$

\noindent These imply that
$$ p_{8}= X_{1}p_{11}- X_{0}p_{12}$$
$$ p_{9}= X_{1}p_{12}- X_{2}p_{11}$$
$$ p_{10}= X_{1}p_{11}- X_{0}p_{12}$$
$$ p_{7}= X_{1}p_{12}- X_{2}p_{11}$$

\noindent Multiplying the sixth row of the matrix $B$ with the column vector and plugging these values 
in the equation obtained we get
$$f_{3} = -X_{2}^{q'}Y^{v-w}p_{12} + X_{0}^{\nu-1}p_{11} + X_{0}^{\lambda-1}p_{1} + X_{2}^{q'}p_{4}.$$

\noindent Therefore, the matrix obtained by putting $f_{i}$ in the $i$-th row is the second syzygy matrix, which is the following:
$$ {\scriptsize
\left[
\begin{array}{cccccc}
0 & 0 & Y^{v-w} & X_{0}^{\mu} & 0 & 0\\
0 & 0 & -X_{2}^{q-q'} & -Y^{w} & 0 & Y^{v} - X_{2}^{q-q'}X_{0}^{\mu} \\
X_{0}^{\lambda-1} & 0 & 0 & X_{2}^{q'} & X_{0}^{\nu-1} & -X_{2}^{q'}Y^{v-w}\\
Y^{v-w} & X_{0}^{\mu}& 0 & 0 & 0 & 0\\
-X_{2} & 0 & X_{1} & 0 & -X_{0}^{\mu}X_{2} & X_{0}^{\mu}X_{1}\\
X_{1} & 0 & -X_{0} & 0 & X_{0}^{\mu}X_{1}& - X_{0}^{\mu +1}\\
-X_{2}^{q-q'} & -Y^{w} & 0 & 0 & Y^{v} - X_{2}^{q-q'}X_{0}^{\mu} & 0\\
0 & X_{2} & 0 & -X_{1} & -Y^{v-w}X_{2} & Y^{v-w}X_{1}\\
0 & -X_{1} & 0 & X_{0} & -Y^{v-w}X_{1} & -Y^{v-w}X_{0}
\end{array}
\right]
}
$$
Denoting the columns by $L_{i}$ from the left we see that 
$$L_{6} = -Y^{v-w}L_{4} +X_{0}^{\mu}L_{3}$$
$$L_{5} = X_{0}^{\mu}L_{1} + Y^{v-w}L_{2}.$$
Therefore, a reduced form of the second syzygy matrix is
$${\scriptsize
C = \left[
\begin{matrix}
0 & 0 & Y^{v-w} & X_{0}^{\mu}\\
0 & 0 & -X_{2}^{q-q'} & -Y^{w}\\
X_{0}^{\lambda-1} & 0 & 0 & X_{2}^{q'}\\
Y^{v-w} & X_{0}^{\mu}& 0 & 0\\
-X_{2} & 0 & X_{1} & 0\\
X_{1} & 0 & -X_{0} & 0\\
-X_{2}^{q-q'}& -Y^{w} & 0 & 0\\
0 & X_{2} & 0 & -X_{1}\\
0 & -X_{1} & 0 & X_{0}
\end{matrix}
\right]
}
$$
It is easy to see that the third syzygy matrix will be the zero matrix. Therefore, 
$$0\longrightarrow R^{4}\stackrel{\theta{_{C}}}{\longrightarrow} R^{9}\stackrel{\theta{_{B}}}{\longrightarrow} R^{6}\stackrel{\theta{_{A}}}{\longrightarrow} R\longrightarrow k[\Gamma]\longrightarrow0$$ is a free resolution of the binomial ideal in question, where $\theta_{A}$, 
$\theta_{B}$, $\theta_{C}$ are the maps given by the matrices $A$, $B$ and $C$ respectively. This resolution is minimal if 
$\mu\neq 0$, $\lambda\neq 1$, $q'\neq 0$, since the entries of the matrices are from the maximal ideal $(X_{0}, X_{1}, X_{2}, Y)$ and 
the total Betti numbers are $[6,9,4]$. We now consider the following cases for which our resolution fails to be minimal and we 
use Corollary \ref{homo2} to extract a minimal one from this.  

\begin{itemize}
\item[(a)] $\mu=0$;
\smallskip

\item[(b)] $\mu\neq 0$,$\lambda=1$;
\smallskip

\item[(c)] $\mu\neq 0$,$\lambda\neq 1$,$q'= 0$;
\end{itemize}
\smallskip
\noindent We notice that in this case $\lambda=1, \mu=0$ is not possible due to \quad {\rm by \quad [2.2, \mbox{\cite{seng1}}]}
\bigskip

\noindent\textbf{Case (a): $\mu=0$}
\smallskip

\noindent We take ${\scriptsize P_{2}=E_{85}(Y^{v-w})E_{84}(X_{2})E_{81}(-X_{1})E_{91}(X_{0})E_{94}(X_{1})E_{96}(Y^{v-w})}$ and 
$P_{3}=E_{42}(Y^{w})E_{62}(-X_{1}X_{2}^{q'})E_{53}(Y^{w})E_{63}(-X_{2}^{q'+1})$. 
The syzygy matrices under this new transformation become 
$$A = {\scriptsize [\xi_{11}, 0, 0, \psi_{0}, \psi_{1}, \theta]},$$

$$B = {\scriptsize \left[\begin{matrix}
0 & -X_{2}^{q'}Y^{v-w} & X_{2}^{q-q'}-Y^{v} & 0 & 0 & 0 & X_{0}^{\lambda-1} & 0 & 0\\
0 & 0 & 0 & 0 & 0 & 0 & 0 & 1 & 0\\
0 & 0 & 0 & 0 & 0 & 0 & 0 & 0 & 1\\
0 & X_{1} & 0 & 0 & Y^{v}-X_{2}^{q-q'} & 0 & -X_{2} & 0 & 0\\
0 & -X_{0} & 0 & 0 & 0 & X_{2}^{q-q'}-Y^{v} & X_{1} & 0 & 0\\
0 & 0 & X_{1}^{2}-X_{2}X_{0} & 0 & X_{1}X_{2}^{q'}Y^{v-w}-X_{0}^{\lambda} & X_{2}^{q'+1}Y^{v-w}-X_{1}X_{0}^{\lambda-1} & 0 & 0 & 0\\
\end{matrix}
\right]},$$

$$C = {\scriptsize\left[\begin{matrix}
0 & 0 & 0 & 1\\
0 & 0 & Y^{v}-X_{2}^{q-q'} & 0\\
X_{0}^{\lambda-1} & 0 & -X_{2}^{q'}Y^{v-w} & 0\\
0 & 1 & 0 & 0\\
-X_{2} & 0 & X_{1} & 0\\
X_{1} & 0 & -X_{0} & 0\\
Y^{v}-X_{2}^{q-q'}& 0 & 0 & 0\\
0 & 0 & 0 & 0\\
0 & 0 & 0 & 0
\end{matrix}\right]}.$$

Finally, applying Lemma \ref{hom3} we get that 
$$A = {\scriptsize[\xi_{11},\psi_{0},\psi_{1},\theta]},$$ 
$$B = {\scriptsize\left[\begin{matrix}
-X_{2}^{q'}Y^{v-w} & X_{2}^{q-q'}-Y^{v} & 0 & 0 & X_{0}^{\lambda-1}\\
X_{1} & 0 & Y^{v}-X_{2}^{q-q'} & 0 & -X_{2}\\
-X_{0} & 0 & 0 & X_{2}^{q-q'}-Y^{v} & X_{1}\\
0 & X_{1}^{2}-X_{2}X_{0} & X_{1}X_{2}^{q'}Y^{v-w}-X_{0}^{\lambda} & X_{2}^{q'+1}Y^{v-w}-X_{1}X_{0}^{\lambda-1} & 0\\
\end{matrix}\right]},$$
$$C = {\scriptsize \left[\begin{matrix}
0 & Y^{v}-X_{2}^{q-q'}\\
X_{0}^{\lambda-1} & -X_{2}^{q'}Y^{v-w}\\
-X_{2} & X_{1}\\
X_{1} & -X_{0}\\
Y^{v}-X_{2}^{q-q'} & 0\\
\end{matrix}\right]}.$$
Therefore, the total Betti numbers in this case are [4,5,2].
\bigskip

\noindent\textbf{Case (b): $\mu\neq 0$, $\lambda=1$} 
\smallskip

\noindent We take $P_{2}=E_{43}(Y^{v-w})E_{53}(-X_{2})E_{63}(X_{1})E_{73}(-X_{2}^{q-q'})$ and 
$P_{1}=E_{14}(-X_{2}^{q'})$. Finally, applying Corollary \ref{homo2} we get the syzygy matrices as 

$$A = {\scriptsize [\xi_{11},\phi_{0}, \phi_{1}, \psi_{0}, \psi_{1},\theta ]},$$

$$B={\scriptsize\left[\begin{matrix}
-X_{2}^{q} & -X_{2}^{q'}Y^{v-w} & Y^{w} & 0 & 0 & X_{0}^{\mu} & 0 & 0\\
X_{1} & 0 & -X_{2} & -Y^{v-w} & 0 & 0 & X_{0}^{\mu} & 0\\
-X_{0} & 0 & X_{1} & 0 & -Y^{v-w} & 0 & 0 & X_{0}^{\mu}\\
0 & X_{1} & 0 & X_{2}^{q-q'} & 0 & -X_{2} & -Y^{w} & 0\\
0 & -X_{0} & 0 & 0 & X_{2}^{q-q'} & X_{1} & 0 & -Y^{w}\\
0 & 0 & 0 & -X_{0} & -X_{1} & 0 & X_{1}X_{2}^{q'} & X_{2}^{q'+1}
\end{matrix}\right]},$$

$$C= {\scriptsize\left[\begin{matrix}
0 & Y^{v-w} & X_{0}^{\mu}\\
0 & -X_{2}^{q-q'} & -Y^{w}\\
X_{0}^{\mu}& 0 &-X_{2}^{q'}Y^{v-w}\\
0 & X_{1} & X_{2}^{q'+1}\\
0 &-X_{0} & -X_{1}X_{2}^{q'}\\
-Y^{w} & 0 & X_{2}^{q}\\
X_{2} & 0 & -X_{1}\\
-X_{1} & 0 & X_{0}\\
\end{matrix}
\right]}.$$
The total Betti numbers are $[6,8,3]$.
\bigskip

\noindent\textbf{Case (c): $\mu\neq 0$, $\lambda\neq 1$, $q'= 0$}
\smallskip

\noindent Finally applying Corollary \ref{homo2} with $P_{2}=E_{23}(-Y^{w})E_{13}(X_{0}^{\mu})E_{83}(-X_{1})E_{93}(X_{0})$ 
and $P_{1}=E_{41}(-X_{0}^{\lambda-1})$.  we get the syzygy matrices as

$$A=[{\scriptsize \xi_{11},\phi_{0},\phi_{1},\psi_{0},\psi_{1},\theta }],$$

$$B= {\scriptsize\left[\begin{matrix}
-X_{2}^{q} & -Y^{v-w} & Y^{w}X_{0}^{\lambda -1} & 0 & 0 & X_{0}^{\nu -1} & 0 & 0 \\
X_{1} & 0 & -X_{2} & -Y^{v-w} & 0 & 0 & X_{0}^{\mu} & 0 \\
-X_{0}& 0  & X_{1} & 0 & -Y^{v-w} & 0 & 0 & X_{0}^{\mu} \\
0 & X_{1} & 0 & X_{2}^{q} & 0 & -X_{2} & -Y^{w} & 0 \\
0 & -X_{0} & 0 & 0 & X_{2}^{q} & X_{1} & 0 & -Y^{w}\\
0 & 0 & 0 & -X_{0}^{\lambda}& -X_{1}X_{0}^{\lambda -1} & 0 & X_{1} & X_{2}
\end{matrix}\right]}$$

$$C = {\scriptsize\left[\begin{matrix}
-X_{0}^{\nu-1} & 0 & Y^{v-w}\\
X_{0}^{\lambda-1}Y^{w} & 0 & -X_{2}^{q}\\
Y^{v-w} & X_{0}^{\mu} & 0\\
-X_{2} & 0 & X_{1}\\
X_{1} & 0 & -X_{0}\\
-X_{2}^{q} & -Y^{w} & 0\\
X_{1}X_{0}^{\lambda-1} & X_{2} & 0\\
-X_{0}^{\lambda} & -X_{1} & 0\\
\end{matrix}\right]}$$
The total Betti numbers are $[6,8,3]$. 



\section{\textbf{$W\neq\emptyset$; $r=2$, $r^{'} = 1$}}                  

\noindent In this case, $\nu = \lambda + \mu +1$. Let 
$\mathcal{G} = \{\xi_{11}, \phi_{0}, \phi_{1}, \psi_{0}, \psi_{1}, \theta\}$, such that                         
\begin{eqnarray*}
\xi_{11} & = & X_{1}^{2} - X_{0}X_{2}\\
\phi_{0} & = & X_{2}^{q+1} - X_{0}^{\lambda}Y^{w}\\
\psi_{0} & = & X_{1}X_{2}^{q'}Y^{v-w} - X_{0}^{\lambda + \mu+1}\\
\psi_{1} & = & X_{2}^{q'+1}Y^{v-w} - X_{0}^{\lambda + \mu }X_{1}\\
\theta & = & Y^{v} - X_{0}^{\mu}X_{1}X_{2}^{q-q'}\\
\end{eqnarray*}
\noindent Let $A=[\xi_{11}, \phi_{0},\psi_{0},\psi_{1}, \theta ]$, which is the 0th syzygy matrix. 
We know that $\mathcal{G}$ is a Gr\"{o}bner Basis with respect to the graded reverse lexicographic 
order from the work done in \cite{seng1}. We indicate below the exact results of \cite{seng1}, which 
have been used together with Theorem \ref{schreyer} for computing the generators for the first syzygy module. 

\begin{eqnarray*}
R_{1} & = &\left(-X_{2}^{q+1}+X_{0}^{\lambda}Y^{w}, X_{1}^{2}-X_{0}X_{2}, 0, 0, 0 \right) \quad {\rm by \quad [3.1, \mbox{\cite{seng1}}]}\\
R_{2} & = &\left(-X_{2}^{q'}Y^{v-w}, 0, X_{1}, -X_{0}, 0\right) \quad {\rm by \quad [7.2, \mbox{\cite{seng1}}]}\\
R_{3} & = &\left(X_{0}^{\lambda + \mu }X_{1}-X_{2}^{q'+1}Y^{v-w}, 0, 0, X_{1}^2-X_{0}X_{2}, 0 \right) \quad {\rm by \quad [3.1, \mbox{\cite{seng1}}]}\\
R_{4} & = &\left(X_{0}^{\mu}X_{1}X_{2}^{q-q'} - Y^{v}, 0, 0, 0, X_{1}^{2}-X_{0}X_{2} \right) \quad {\rm by \quad [3.1, \mbox{\cite{seng1}}]}\\
R_{5} & = & \left(-X_{0}^{\lambda + \mu}X_{2}^{q-q'}, -X_{1}Y^{v-w}, X_{2}^{q-q'+1}, 0, -X_{0}^{\lambda}X_{1} \right) \quad {\rm by \quad [8.6, \mbox{\cite{seng1}}]}\\
R_{6} & = &\left(0, -Y^{v-w}, 0, X_{2}^{q-q'}, -X_{0}^{\lambda} \right) \quad {\rm by \quad [8.1, \mbox{\cite{seng1}}]}\\
R_{7} & = &\left(0, X_{0}^{\mu}X_{1}X_{2}^{q-q'}-Y^{v}, 0, 0, X_{2}^{q+1}-X_{0}^{\lambda}Y^{w}\right) \quad {\rm by \quad [3.1, \mbox{\cite{seng1}}]}\\
R_{8} & = &\left(X_{0}^{\lambda + \mu}, 0, -X_{2}, X_{1}, 0 \right) \quad {\rm by \quad [4.5, \mbox{\cite{seng1}}]}\\
R_{9} & = &\left(X_{0}^{\mu}X_{2}^{q}, X_{0}^{\mu+1}, -Y^{w}, 0, X_{1}X_{2}^{q'}\right) \quad {\rm by \quad [9.1, \mbox{\cite{seng1}}]}\\
R_{10} & = &\left(0, X_{0}^{\mu}X_{1}, 0, -Y^{w}, X_{2}^{q'+1}\right) \quad {\rm by \quad [9.1, \mbox{\cite{seng1}}]}\\
\end{eqnarray*}

\noindent We observe that 
\begin{eqnarray*}
R_{3} & = & X_{2} \cdot R_{2} + X_{1}\cdot R_{8} \\
R_{5} & = & X_{1}\cdot R_{6}  - X_{2}^{q-q'}\cdot R_{8}\\
R_{7} & = & Y^{w} \cdot R_{6} + X_{2}^{q-q'}\cdot R_{10}\\
\end{eqnarray*}

\noindent We remove $R_{3}$, $R_{7}$, $R_{5}$ from the list and the first syzygy matrix is given by the matrix
$$B= {\scriptsize\left[\begin{matrix}
-X_{2}^{q+1} + X_{0}^{\lambda}Y^{w} & -X_{2}^{q'}Y^{v-w} & X_{0}^{\mu}X_{1}X_{2}^{q-q'}-Y^{v} & 0 & X_{0}^{\lambda + \mu } & 
X_{0}^{\mu}X_{2}^{q} & 0\\
X_{1}^{2}-X_{0}X_{2} & 0 & 0 & -Y^{v-w} & 0 & X_{0}^{\mu+1} & X_{0}^{\mu}X_{1}\\
0 & X_{1} & 0 & 0 & -X_{2} & -Y^w & 0\\
0 & -X_{0} & 0 & X_{2}^{q-q'} & X_{1} & 0 &-Y^w\\
0 & 0 & X_{1}^2-X_{0}X_{2} & -X_{0}^{\lambda} & 0 & X_{1}X_{2}^{q'} & X_{2}^{q'+1}
\end{matrix}
\right]}$$
In order to determine the second syzygy matrix, we consider the kernel of the map given by the matrix $B$. 
Let $[f_{1},f_{2},f_{3},f_{4},f_{5},f_{6},f_{7}]$ be an element of the kernel of $B$. Then
$${\scriptsize\left[\begin{matrix}
-X_{2}^{q+1} + X_{0}^{\lambda}Y^{w} & -X_{2}^{q'}Y^{v-w} & X_{0}^{\mu}X_{1}X_{2}^{q-q'}-Y^{v} & 0 & X_{0}^{\lambda + \mu } 
& X_{0}^{\mu}X_{2}^{q} & 0\\
X_{1}^{2} - X_{0}X_{2} & 0 & 0 & -Y^{v-w} & 0 & X_{0}^{\mu+1} & X_{0}^{\mu}X_{1}\\
0 & X_{1} & 0 & 0 & -X_{2} & -Y^{w} & 0\\
0 & -X_{0} & 0 & X_{2}^{q-q'} & X_{1} & 0 &-Y^{w}\\
0 & 0 & X_{1}^2-X_{0}X_{2} & -X_{0}^{\lambda} & 0 & X_{1}X_{2}^{q'} & X_{2}^{q'+1}
\end{matrix}\right]
\left[\begin{matrix}
f_{1}\\
f_{2}\\
f_{3}\\
f_{4}\\
f_{5}\\
f_{6}\\
f_{7}
\end{matrix}\right]=0}$$

\noindent Multiplication of 3rd row with the column vector gives 

\begin{equation}\label{eq:1}
X_{1}f_{2}-X_{2}f_{5}-Y^{w}f_{6}=0 
\end{equation}
\smallskip

\noindent This implies that $f_{6}\in \langle X_{1},X_{2} \rangle $, since the polynomials 
$X_{1}, X_{2}, Y^{w}$ form a regular sequence. 
We may write $f_{6}=X_{1}p_{1}+X_{2}p_{2}$, and therefore 
$$X_{1}(f_{2}-Y^{w}p_{1})= X_{2}(f_{5}+Y^{w}p_{2}).$$
\noindent This shows that $X_{1}\mid (f_{5}+Y^{w}p_{2})$. Taking the quotient as $p_{3}$ we get,
$$f_{2}=X_{2}p_{3}+p_{1}Y^{w},$$
$$f_{5}=X_{1}p_{3}-p_{2}Y^{w},$$
where $p_{i}\in R$. 
\medskip

Multiplication of 4th row with the column vector gives 
\begin{equation}\label{eq:2}
-X_{0}f_{2}+X_{2}^{q-q'}f_{4}+X_{1}f_{5}-Y^{w}f_{7}=0.
\end{equation}
\smallskip

\noindent Plugging in values of $f_{2}$ and $f_{5}$ in the above equation we get, 
$$X_{2}^{q-q'}f_{4}-Y^{w}(f_{7}+X_{0}p_{1}+X_{1}p_{2})=-(X_{1}^2-X_{0}X_{2})p_{3}.$$
We see that $Y^{w}\mid [X_{2}^{q-q'}f_{4}+(X_{1}^2-X_{0}X_{2})p_{3}]$. We may therefore write 
\begin{equation}\label{eq:3}
X_{2}^{q-q'}f_{4}= -(X_{1}^{2} - X_{0}X_{2})p_{3} + Y^{w}p_{4}
\end{equation}
$$f_{7}=-X_{0}p_{1}-X_{1}p_{2}+p_{4}.$$

Multiplication of 6th row with the column vector gives 
\begin{equation}\label{eq:4}
(X_{1}^{2}-X_{0}X_{2})f_{3}-X_{0}^{\lambda}f_{4}+X_{1}X_{2}^{q'}f_{6}+X_{2}^{q'+1}f_{7}=0
\end{equation}
\smallskip

\noindent Multiplying $(5.4)$ by $X_{2}^{q-q'}$ and plugging in the value of $X_{2}^{q-q'}f_{4}$ obtained from $(5.3)$, we get 
\begin{equation}\label{eq:5}
(X_{1}^2-X_{0}X_{2})(X_{2}^{q-q'}f_{3}+X_{0}^{\lambda}p_{3}+p_{1}X_{2}^{q})=(X_{2}^{q'+1}-X_{0}^{\lambda}Y^{w})X_{2}^{q-q'}p_{4}.
\end{equation} 
\smallskip

\noindent Therefore, $(X_{1}^2-X_{0}X_{2})\mid p_{4}$. Let $p_{4}=(X_{1}^2-X_{0}X_{2})p_{5}$. 
Putting this value of $p_{4}$ in $(5.3)$ we get 
\begin{equation}\label{eq:6}
X_{2}^{q-q'}f_{4}=(X_{1}^2-X_{0}X_{2})(Y^wp_{5}-p_{3}).
\end{equation}
\smallskip

\noindent Therefore, $(X_{1}^2-X_{0}X_{2})\mid f_{4}$ and we may write $f_{4}=(X_{1}^2-X_{0}X_{2})p_{6}$. 
From $(5.6)$ we get 
$$p_{3}=Y^{w}p_{5}-X_{2}^{q-q'}p_{6}.$$
Plugging these values of $p_{3}$ and $p_{4}$ in $(5.5)$, we get 
\begin{equation}\label{eq:7} 
f_{3}=X_{0}^{\lambda}p_{6}-X_{2}^{q'}p_{1}-X_{2}^{q'+1}p_{5}.
\end{equation}
\smallskip

Multiplication of 2nd row with the column vector gives 
$$(X_{1}^{2} - X_{0}X_{2})f_{1} - Y^{v-w}f_{4} + X_{0}^{\mu+1}f_{6} + X_{0}^{\mu}X_{1}f_{7}=0.$$
Plugging in values of $f_{4}$, $f_{6}$ and $f_{7}$ in the above equation we get,
$$f_{1} = Y^{v-w}p_{6}+p_{2}X_{0}^{\mu}-X_{0}^{\mu}X_{1}p_{5}.$$
Therefore, 
$${\scriptsize\left[\begin{matrix}
0 & X_{0}^{\mu} & -X_{0}^{\mu}X_{1} & Y^{v-w}\\
Y^{w} & 0 & X_{2}Y^{w} & -X_{2}^{q-q'+1}\\
-X_{2}^{q'} & 0 & -X_{2}^{q'+1} & X_{0}^{\lambda}\\
0 & 0 & 0 & (X_{1}^{2}-X_{0}X_{2})\\
0 & -Y^{w} & X_{1}Y^{w} & -X_{1}X_{2}^{q-q'}\\
X_{1} & X_{2} & 0 & 0\\
-X_{0} & -X_{1} & X_{1}^{2}-X_{0}X_{2} & 0
\end{matrix}\right]}
{\scriptsize\left[\begin{matrix}
p_{1}\\p_{2}\\p_{5}\\p_{6}
\end{matrix}\right]}
{\scriptsize =\left[
\begin{array}{c}
f_{1}\\f_{2}\\f_{3}\\f_{4}\\f_{5}\\f_{6}\\f_{7}
\end{array}\right]}$$
Hence, the 2nd syzygy matrix is given by
$${\scriptsize \left[\begin{matrix}
0 & X_{0}^{\mu} & -X_{0}^{\mu}X_{1} & Y^{v-w}\\
Y^{w} & 0 & X_{2}Y^{w} & -X_{2}^{q-q'+1}\\
-X_{2}^{q'} & 0 & -X_{2}^{q'+1} & X_{0}^{\lambda}\\
0 & 0 & 0 & (X_{1}^2-X_{0}X_{2})\\
0 & -Y^{w} & X_{1}Y^{w} & -X_{1}X_{2}^{q-q'}\\
X_{1} & X_{2} & 0 & 0\\
-X_{0} & -X_{1} & X_{1}^2-X_{0}X_{2} & 0
\end{matrix}\right]}$$

\noindent We observe that $L_{3}= -X_{1}L_{2} + X_{2}L_{1}$, where $L_{i}$ denotes the ith column of the above matrix. 
Removing the third column we get the reduced form of the second syzygy matrix
$${\scriptsize C=\left[
\begin{matrix}
0 & X_{0}^{\mu} & Y^{v-w}\\
Y^{w} & 0 & -X_{2}^{q-q'+1}\\
-X_{2}^{q'} & 0 & X_{0}^{\lambda}\\
0 & 0 & X_{1}^2-X_{0}X_{2}\\
0 & -Y^{w} & -X_{1}X_{2}^{q-q'}\\
X_{1} & X_{2} & 0\\
-X_{0} & -X_{1} & 0
\end{matrix}
\right]}$$
Now we calculate the third syzygy matrix. Let $[g_{1},g_{2},g_{3}]$ be an element in the kernel of $C$. This means that
$${\scriptsize \left[\begin{matrix}
0 & X_{0}^{\mu} & Y^{v-w}\\
Y^{w} & 0 & -X_{2}^{q-q'+1}\\
-X_{2}^{q'} & 0 & X_{0}^{\lambda}\\
0 & 0 & X_{1}^2-X_{0}X_{2}\\
0 & -Y^{w} & -X_{1}X_{2}^{q-q'}\\
X_{1} & X_{2} & 0\\
-X_{0} & -X_{1} & 0
\end{matrix}
\right]
\left[
\begin{matrix}
g_{1}\\ g_{2}\\ g_{3}
\end{matrix}\right]}=0.$$
Multiplication of the 4th row with the column vector gives $(X_{1}^2-X_{0}X_{2})g_{3}=0$, and 
therefore $g_{3}=0$. Similarly, multiplication of the 6th and 7th row with the column vector gives 
$X_{1}g_{1}-X_{2}g_{2}=0$ and $-X_{2}g_{1}+X_{1}g_{2}=0$, implying that $g_{1}=g_{2}=0$. Therefore, 
the map given by the matrix $C$ is injective. Hence, we get that 
$$0\longrightarrow R^{3}\stackrel{^\theta{_{C}}}{\longrightarrow} R^{7}\stackrel{^\theta{_{B}}}{\longrightarrow} R^{5}\stackrel{^\theta{_{A}}}{\longrightarrow} R\longrightarrow k[\Gamma]\longrightarrow0$$ 
is a free resolution of the binomial ideal in question, where, $\theta_{A},\theta_{B},\theta_{C},$ are 
the maps given by the matrices $A$, $B$ and $C$ respectively. This resolution is minimal if 
$\mu\neq 0$, $q\neq q'$, $q'\neq 0$, for, the entries of the syzygy matrices would then belong to 
the maximal ideal $(X_{0},X_{1},X_{2},Y)$ and the total Betti numbers would be $[5,7,3]$. 
We now consider the following cases for which our resolution fails to be minimal and we use Corollary \ref{homo2}
to extract a minimal one out of this. 

\begin{itemize}
\item[(a)] $\mu=0,q'\neq 0$
\smallskip
\item[(b)] $\mu=0,q'=0$
\smallskip
\item[(c)] $\mu\neq 0,q\neq q',q'=0$ 
\smallskip
\item[(d)] $\mu\neq 0,q=q',q'\neq 0$
\end{itemize}
\bigskip

\noindent\textbf{Case (a): $\mu=0,q'\neq 0$}
\smallskip

\noindent We apply Corollary \ref{homo2} by taking $P_{2}=E_{51}(-Y^{w})E_{61}(X_{2})E_{71}(-X_{1})$ 
and $P_{1}=E_{23}(-Y^{v-w})$. Finally, applying Lemma \ref{hom3} we get the syzygy matrices as 
$$ A = [{\scriptsize \xi_{11},\phi_{0},\psi_{0},\psi_{1},\theta }]$$
$${\scriptsize
B = \left[\begin{matrix}
-X_{2}^{q'}Y^{v-w} & X_{1}X_{2}^{q-q'}-Y^{v} & 0 & X_{0}^{\lambda} & X_{2}^{q} & 0\\
0 & 0 & -Y^{v-w} & 0 & X_{0} & X_{0}X_{1}\\
X_{1} & 0 & 0 & -X_{2} & -Y^{w} & 0\\
-X_{0} & 0 & X_{2}^{q-q'} & X_{1} & 0 &-Y^{w}\\
0 & X_{1}^{2}-X_{0}X_{2} & -X_{0}^{\lambda} & 0 & X_{1}X_{2}^{q'} & X_{2}^{q'+1}
\end{matrix}\right]}$$
$$C = {\scriptsize\left[
\begin{matrix}
Y^{w} & -X_{2}^{q-q'+1}\\
-X_{2}^{q'} & X_{0}^{\lambda}\\
0 & X_{1}^2-X_{0}X_{2}\\
0 & Y^{v}-X_{1}X_{2}^{q-q'}\\
X_{1} & -X_{2}Y^{v-w}\\
-X_{0} & X_{1}Y^{v-w}
\end{matrix}
\right]}$$

\noindent The total Betti numbers are $[5,6,2]$. 
\bigskip

\noindent\textbf{Case (b): $\mu=0$, $q'= 0$}
\smallskip

\noindent Here along with the preceding case the entry $(2,1)$ of the matrix C as obtained in Case (a) will be affected. 
Hence, we take  $P_{2}=E_{12}(-Y^w)E_{52}(-X_{1})E_{62}(X_{0})$ and $P_{1}=E_{12}(X_{0}^{\lambda})$ and apply Corollary \ref{homo2} to 
get the syzygy matrices as
$$A=[ {\scriptsize \xi_{11},\phi_{0},\psi_{0},\psi_{1},\theta }]$$
$$B {\scriptsize\left[
\begin{matrix}
-X_{2}^{q'}Y^{v-w} & 0 & X_{0}^{\lambda } & X_{2}^{q} & 0\\
0 & -Y^{v-w} & 0 & X_{0}^{1} & X_{0}X_{1}\\
X_{1} & 0 &-X_{2} &-Y^{w} & 0\\
-X_{0} & X_{2}^{q-q'} & X_{1} & 0 & -Y^{w}\\
0 & -X_{0}^{\lambda} & 0 & X_{1}X_{2}^{q'} & X_{2}^{q'+1}
\end{matrix}\right]}$$
$$C = {\scriptsize \left[
\begin{array}{c}
X_{0}^{\lambda}Y^{w}-X_{2}^{q-q'+1}\\
X_{1}^2-X_{0}X_{2}\\
Y^{v}-X_{1}X_{2}^{q-q'}\\
X_{0}^{\lambda}X_{1}-X_{2}Y^{v-w}\\
X_{1}Y^{v-w}-X_{0}^{\lambda+1}
\end{array}
\right]}$$

\noindent The total Betti numbers are $[5,5,1]$. 
\bigskip

\noindent\textbf{Case (c): $\mu\neq 0$, $q'\neq q$, $q'=0$}
\smallskip

\noindent We take $P_{2}= E_{23}(-Y^{w})E_{63}(-X_{1})E_{73}(X_{0})$ and $P_{1}=E_{13}(X_{0}^{\lambda})$. Finally, 
applying Corollary \ref{homo2} we get the syzygy matrices as 
$$A=[{\scriptsize \xi_{11},\phi_{0},\psi_{0},\psi_{1},\theta }]$$
$${\scriptsize B=\left[\begin{matrix}
-X_{2}^{q+1}+X_{0}^{\lambda}Y^{w} & -Y^{v-w} & 0 & X_{0}^{\lambda } & X_{2}^{q} & 0\\
X_{1}^2-X_{0}X_{2} & 0 & -Y^{v-w} & 0 & X_{0} & X_{0}X_{1}\\
0 & X_{1} & 0 &-X_{2} & -Y^{w} & 0\\
0 & -X_{0} & X_{2}^{q} & X_{1} & 0 &-Y^{w}\\
0 & 0 &-X_{0}^{\lambda} & 0 & X_{1} & X_{2}
\end{matrix}\right]}$$
$$C= {\scriptsize\left[\begin{matrix}
X_{0}^{\mu} & Y^{v-w}\\
0 & X_{0}^{\lambda}Y^w-X_{2}^{q+1}\\
0 & X_{1}^2-X_{0}X_{2}\\
-Y^{w} & -X_{1}X_{2}^{q}\\
X_{2} & X_{0}^{\lambda}X_{1}\\
-X_{1} & -X_{0}^{\lambda+1}
\end{matrix}\right]}$$

\noindent The total Betti numbers are [5,6,2]. 
\bigskip

\noindent\textbf{Case (d): $\mu\neq 0$, $q'=q $, $q'\neq0$}
\smallskip

\noindent We apply Corollary \ref{homo2} by taking $P_{3}=E_{54}(X_{0}^{\lambda})E_{24}(Y^{v-w})$ and 
$P_{2}=E_{42}(X_{0})E_{45}(-X_{1})E_{47}(Y^w)$. Finally applying Lemma \ref{hom3} we get the 
syzygy matrices as   
$$A=[{\scriptsize \xi_{11},\phi_{0},\psi_{0},\theta }]$$
$${\scriptsize B=\left[\begin{matrix}
-X_{2}^{q+1}+X_{0}^{\lambda}Y^{w} & -X_{2}^{q}Y^{v-w} & X_{0}^{\mu}X_{1}-Y^{v} & X_{0}^{\lambda + \mu } 
& X_{0}^{\mu}X_{2}^{q} & 0\\
X_{1}^2-X_{0}X_{2} & -X_{0}Y^{v-w} & 0 & X_{1}Y^{v-w} & X_{0}^{\mu+1} & X_{0}^{\mu}X_{1}-Y^{v}\\
0 & X_{1} & 0 & -X_{2} & -Y^w & 0\\
0 & -X_{0}^{\lambda+1} & X_{1}^2-X_{0}X_{2} & X_{1}X_{0}^{\lambda} & X_{1}X_{2}^{q} & X_{2}^{q+1}-X_{0}^{\lambda}Y^{w}
\end{matrix}\right]}$$
$${\scriptsize C=\left[\begin{matrix}
0 & X_{0}^{\mu} & Y^{v-w}\\
Y^{w} & 0 & -X_{2}\\
-X_{2}^{q} & 0 & X_{0}^{\lambda}\\
0 & -Y^{w} & -X_{1}\\
X_{1} & X_{2} & 0\\
-X_{0} & -X_{1} & 0
\end{matrix}\right]}$$

\noindent The total Betti numbers are $[4,6,3]$.



\section{\textbf{$W\neq\emptyset$; $r=r'= 2$}} 

\noindent In this case, because of $r=r'$, we have $\nu = \lambda+\mu $ and $q>q'$. 
Let $\mathcal{G}=\{\xi_{11},\phi_{0},\psi_{0},\theta\}$, such that 
\begin{eqnarray*}
\xi_{11} & = & X_{1}^2-X_{0}X_{2}\\
\phi_{0} & = & X_{2}^{q+1} - X_{0}^{\lambda}Y^{w}\\
\psi_{0} & = & X_{2}^{q'+1}Y^{v-w} - X_{0}^{\lambda + \mu}\\
\theta & = & Y^{v} - X_{2}^{q-q'}X_{0}^{\mu}\\
\end{eqnarray*}
\noindent Let $A=[\xi_{11}, \phi_{0},\psi_{0},\theta ]$ be the $0$-th syzygy matrix. We know that 
$\mathcal{G}$ is a Gr\"{o}bner Basis with respect to the graded reverse lexicographic order from 
the work done in \cite{seng1}. We indicate below the exact results of \cite{seng1}, which have been 
used together with \ref{schreyer} for computing the generators for the first syzygy module . 

\begin{eqnarray*}
R_{1} & = &\left(X_{0}^{\lambda}Y^{w}-X_{2}^{q+1}, X_{1}^{2}-X_{0}X_{2}, 0, 0\right) {\rm by \quad [3.1, \mbox{\cite{seng1}}]}\\
R_{2} & = &\left(X_{0}^{\lambda+\mu}-X_{2}^{q'+1}Y^{v-w}, 0, X_{1}^{2}-X_{0}X_{2}, 0\right){\rm by \quad [3.1, \mbox{\cite{seng1}}]}\\
R_{3} & = &\left(X_{0}^{\mu}X_{2}^{q-q'}-Y^{v}, 0, 0, X_{1}^{2}-X_{0}X_{2}\right){\rm by \quad [3.1, \mbox{\cite{seng1}}]}\\
R_{4} & = &\left(0,-Y^{v-w},X_{2}^{q-q'},-X_{0}^{\lambda}\right){\rm by \quad [8.1, \mbox{\cite{seng1}}]}\\
R_{5} & = &\left(0,X_{0}^{\mu}X_{2}^{q-q'}-Y^{v},0,X_{2}^{q+1}-X_{0}^{\lambda}Y^{w}\right){\rm by \quad [3.1, \mbox{\cite{seng1}}]}\\
R_{6} & = &\left(0,X_{0}^{\mu}, -Y^{w}, X_{2}^{q'+1} \right) {\rm by \quad [9.2, \mbox{\cite{seng1}}]}
\end{eqnarray*}
We observe that $R_{5}=X_{2}^{q-q'}\cdot R_{6} + Y^{w}\cdot R_{4}$. Therefore, after removing $R_{5}$ from our list we get our 
first syzygy matrix
\begin{eqnarray*}
B & = & \left[^{t}R_{1},^{t}R_{2},^{t}R_{3},^{t}R_{4},^{t}R_{6}\right]\\
{} & = &{\scriptsize
\left[\begin{matrix}
X_{0}^{\lambda}Y^{w}-X_{2}^{q+1} & X_{0}^{\lambda+\mu}-X_{2}^{q^{'}+1}Y^{v-w} & X_{0}^{\mu}X_{2}^{q-q^{'}}-Y^{v} & 0 & 0\\
X_{1}^{2}-X_{0}X_{2} & 0 & 0 &-Y^{v-w} & X_{0}^{\mu}\\
0 & X_{1}^{2}-X_{0}X_{2} & 0 & X_{2}^{q-q^{'}} & -Y^{w}\\
0 & 0 & X_{1}^{2}-X_{0}X_{2} & -X_{0}^{\lambda} & X_{2}^{q^{'}+1}\\
\end{matrix}\right]}
\end{eqnarray*}

\noindent We now the determine the second syzygy, which is the kernel of the map given by the matrix B.
Let $[f_{1},f_{2},f_{3},f_{4},f_{5}]\in R^{9}$ be an element of the second syzygy,that is, 
$$B= {\scriptsize\left[
\begin{matrix}
X_{0}^{\lambda}Y^{w}-X_{2}^{q+1} & X_{0}^{\lambda+\mu}-X_{2}^{q^{'}+1}Y^{v-w} & X_{0}^{\mu}X_{2}^{q-q^{'}}-Y^{v} & 0 & 0\\
X_{1}^{2}-X_{0}X_{2} & 0 & 0 &-Y^{v-w} & X_{0}^{\mu}\\
0 & X_{1}^{2}-X_{0}X_{2} & 0 & X_{2}^{q-q^{'}} & -Y^{w}\\
0 & 0 & X_{1}^{2}-X_{0}X_{2} & -X_{0}^{\lambda} & X_{2}^{q^{'}+1}\\
\end{matrix}\right]}
{\scriptsize\left[\begin{matrix}
f_{1}\\
f_{2}\\
f_{3}\\
f_{4}\\
f_{5}
\end{matrix}
\right] =0}.$$

Multiplying the 2nd row of $B$ with the column vector gives 
$$(X_{1}^{2}-X_{0}X_{2})f_{1}-Y^{v-w}f_{4}+X_{0}^{\mu}f_{5}=0$$ 
and it follows that $Y^{v-w}f_{4}\in \langle (X_{1}^{2}-X_{0}X_{2}), X_{0}^{\mu}\rangle$. 
Therefore, 
$$f_{4}\in \langle (X_{1}^{2}-X_{0}X_{2}),X_{0}^{\mu}\rangle.$$
since $X_{1}^{2}-X_{0}X_{2}, X_{0}^{\mu}, Y^{v-w}$ form a regular sequence in $R$. 
Therefore, we may write 
$$f_{4}=(X_{1}^{2}-X_{0}X_{2})p_{1}+(X_{0}^{\mu})p_{2}$$ 
for polynomials $p_{1}$, $p_{2}$, and it follows that,
$$f_{1}=p_{1}Y^{v-w}+X_{0}^{\mu}p_{3},$$

$$f_{5}=Y^{v-w}p_{2}-(X_{1}^{2}-X_{0}X_{2})p_{3}.$$

Multiplying the 3rd row of $B$ with the column vector gives 
$$(X_{1}^{2}-X_{0}X_{2})f_{2}+X_{2}^{q-q'}f_{4}-Y^{w}f_{5}=0.$$
Plugging in the value of $f_{4}$ and $f_{5}$ in the above equation, we get
$$(X_{1}^{2}-X_{0}X_{2})(f_{2}+X_{2}^{q-q'}p_{1}+Y^{w}p_{3})=(Y^{v}-X_{0}^{\mu}X_{2}^{q-q^{'}})p_{2}.$$
This implies that 
$(X_{1}^{2}-X_{0}X_{2})\mid p_{2}$ and $(Y^{v}-X_{0}^{\mu}X_{2}^{q-q^{'}})\mid(f_{2}+X_{2}^{q-q'}p_{1}+Y^{w}p_{3}).$ 
Taking either quotient as $p_{4}$, we get 
$$f_{2}=-X_{2}^{q-q'}p_{1}-Y^{w}p_{3}+(Y^{v}-X_{0}^{\mu}X_{2}^{q-q^{'}})p_{4}$$
$$p_{2}=(X_{1}^{2}-X_{0}X_{2})p_{4}.$$

Multiplying the 4th row of $B$ with the column vector and plugging in the values of 
$f_{4}$ and $f_{5}$ gives
$$f_{3}=X_{0}^{\lambda}p_{1} + (X_{0}^{\lambda+\mu} - X_{2}^{q'+1}y^{v-w})p_{4} + X_{2}^{q'+1}p_{3}.$$
\noindent Therefore,

$$ {\scriptsize\left[\begin{matrix}
Y^{v-w} & X_{0}^{\mu} & 0\\
-X_{2}^{q-q'} & -Y^{w} & Y^{v}-X_{0}^{\mu}X_{2}^{q-q'}\\
X_{0}^{\lambda} & X_{2}^{q'+1} & X_{0}^{\lambda+\mu}-X_{2}^{q'+1}Y^{v-w}\\
X_{1}^{2}-X_{0}X_{2} & 0 & X_{0}^{\mu}(X_{1}^{2}-X_{0}X_{2})\\
0 & -(X_{1}^2-X_{0}X_{2}) & Y^{v-w}(X_{1}^2-X_{0}X_{2})\\
\end{matrix}\right]}
{\scriptsize\left[
\begin{matrix}
p_{1}\\p_{3}\\p_{4}
\end{matrix}\right]}=
{\scriptsize\left[ 
\begin{matrix}
f_{1}\\f_{2}\\f_{3}\\f_{4}\\f_{5}
\end{matrix}\right]}$$
\noindent Hence the second syzygy matrix is given by 
$${\scriptsize\left[\begin{matrix}
Y^{v-w} & X_{0}^{\mu} & 0\\
-X_{2}^{q-q'} & -Y^{w} & Y^{v}-X_{0}^{\mu}X_{2}^{q-q'}\\
X_{0}^{\lambda} & X_{2}^{q'+1} & X_{0}^{\lambda+\mu}-X_{2}^{q'+1}Y^{v-w}\\
X_{1}^2-X_{0}X_{2} & 0 & X_{0}^{\mu}(X_{1}^2-X_{0}X_{2})\\
0 &-(X_{1}^2-X_{0}X_{2}) & Y^{v-w}(X_{1}^2-X_{0}X_{2})\\
\end{matrix}\right]}$$
\noindent Denoting the columns of the above matrix by $L_{i}$ from the left we see that 
$$L_{3}= X_{0}^{\mu} \cdot L_{1} - Y^{v-w} \cdot L_{2}.$$
Therefore, a reduced form of the second syzygy matrix is 
$${\scriptsize
C=\left[\begin{matrix}
Y^{v-w} & X_{0}^{\mu}\\
-X_{2}^{q-q'} & -Y^{w}\\
X_{0}^{\lambda} & X_{2}^{q'+1}\\
X_{1}^{2}-X_{0}X_{2} & 0\\
0 & -(X_{1}^{2}-X_{0}X_{2})\\
\end{matrix}
\right]}$$
It is easy to see that the third syzygy matrix will be the zero matrix. Therefore,  
$$0\longrightarrow R^{2}\stackrel{\theta{_{C}}}{\longrightarrow} R^{5}\stackrel{\theta{_{B}}}{\longrightarrow} R^{4}\stackrel{\theta{_{A}}}{\longrightarrow} R\longrightarrow k[\Gamma]\longrightarrow0$$ is a free resolution of the binomial ideal in question, where $\theta_{A}$, 
$\theta_{B}$, $\theta_{C}$ are the maps given by the matrices $A$, $B$ and $C$ respectively. This resolution is minimal if $\mu\neq 0$, since the entries of the matrices are from the maximal ideal 
$(X_{0}, X_{1}, X_{2}, Y)$ and the total Betti numbers are $[4,5,2]$. 
\medskip

We now consider the case $\mu = 0$. If we take $P_{2}=E_{51}(-(X_{1}^2-X_{0}X_{2}))E_{54}(Y^{v-w})$ and $P_{3}=E_{42}(-X_{2}^{q^{'}+1})$ and apply Corollary \ref{homo2}, we get the syzygy matrices in their minimal forms as 
$$A=[{\scriptsize\xi_{11},\psi_{0},\theta}]$$
$$B={\scriptsize\left[
\begin{matrix}
-\psi_{0} & -\theta & 0\\
\xi_{11} & 0 & -\theta\\
0 & \xi_{11} & -\psi_{0}
\end{matrix}
\right]}$$
$$ C={\scriptsize\left[
\begin{matrix}
\theta\\
-\psi\\ 
\xi_{11}
\end{matrix}\right]}.$$

\noindent The total Betti numbers are $[3,3,1]$.



\section{\textbf{$W = \emptyset$}}

We consider four subcases:
\begin{enumerate}
\item[(a)] $r=1$, $r'=2$; 
\item[(b)] $r=r'=1$;
\item[(c)] $r=2$.
\end{enumerate}

\subsection{\textbf{Case (a): $r=1$, $r'=2$}.} 

In this case 
\begin{eqnarray*}
\xi_{11} & = & X_{1}^2-X_{0}X_{2}\\
\phi_{0} & = & X_{1}X_{2}^{q} - X_{0}^{\lambda}Y^{w}\\
\phi_{1} & = & X_{2}^{q'+1} - X_{0}^{\lambda -1}X_{1}Y^{w}\\
\theta & = & Y^{v} - X_{0}^{\mu}X_{1}X_{2}^{q-q'-1}\\
\end{eqnarray*}
\noindent Let $A=[\xi_{11}, \phi_{0}, \phi_{1}, \theta ]$ be the $0$-th syzygy matrix. We know that 
$\mathcal{G}$ is a Gr\"{o}bner Basis with respect to the graded reverse lexicographic order from the work done in 
\cite{seng1}. We indicate below the exact results of \cite{seng1}, which have been used together with \ref{schreyer} 
for computing the generators for the first syzygy module . 

\begin{eqnarray*}
R_{1} & = &\left(-X_{2}^{q}, X_{1}, -X_{0}, 0\right) {\rm by \quad [6.1, \mbox{\cite{seng1}}]}\\
R_{2} & = &\left(X_{0}^{\lambda - 1}X_{1}Y^{w} -X_{2}^{q+1}, 0, X_{1}^{2}-X_{0}X_{2}, 0\right){\rm by \quad [3.1, \mbox{\cite{seng1}}]}\\
R_{3} & = &\left(X_{0}^{\mu}X_{1}X_{2}^{q-q'-1}-Y^{v}, 0, 0, X_{1}^{2}-X_{0}X_{2}\right){\rm by \quad [3.1, \mbox{\cite{seng1}}]}\\
R_{4} & = &\left(X_{0}^{\lambda - 1}Y^{w}, -X_{2}, X_{1}, 0\right){\rm by \quad [4.5, \mbox{\cite{seng1}}]}\\
R_{5} & = &\left(0, X_{0}^{\mu}X_{1}X_{2}^{q-q'-1}-Y^{v}, 0, X_{1}X_{2}^{q} - X_{0}^{\lambda}Y^{w}\right){\rm by \quad [3.1, \mbox{\cite{seng1}}]}\\
R_{6} & = &\left(0, 0, X_{0}^{\mu}X_{1}X_{2}^{q-q'-1}-Y^{v}, X_{2}^{q'+1} - X_{0}^{\lambda -1}X_{1}Y^{w}, \right) {\rm by \quad [3.1, \mbox{\cite{seng1}}]}
\end{eqnarray*}
We note that $R_{2} = X_{1}\cdot R_{4} + X_{2}\cdot R_{1}$. Therefore, the syzygy matrices are 
$$A=[\xi_{11},\phi_{0},\phi_{1},\theta],$$
$$B={\scriptsize\left[
\begin{matrix}
-X_{2}^{q} & X_{0}^{\mu}X_{2}^{q-q'-1}-Y^{v}& X_{0}^{\lambda-1}Y^{w}  & 0 & 0\\
X_{1} & 0 & -X_{2}  & X_{0}^{\mu}X_{2}^{q-q'-1}-Y^{v} & 0\\
-X_{0}& 0 & X_{1}  & 0 & X_{0}^{\mu}X_{2}^{q-q'-1}-Y^{v}\\
0 & X_{1}^2-X_{0}X_{2}&0  & X_{1}X_{2}^{q}-X_{0}^{\lambda}Y^{w} & X_{2}^{q+1}-X_{0}^{\lambda-1}X_{1}Y^{w}
\end{matrix}\right]}.$$

\noindent In order to determine the second syzygy we proceed to determine the kernel of the map given by the matrix $B$. 
Suppose that 
$$B\dot {\scriptsize\left[
\begin{array}{c}
 f{1}\\f_{2}\\f_{3}\\f_{4}\\f_{5}
\end{array}
\right]
}=0$$
Multiplying the second row of $B$ with the column vector we get 
\begin{equation}\label{ eq1}
X_{1}f_{1}- X_{2}f_{3} -(Y^{v} - X_{0}^{\mu}X_{1}X_{2}^{q-q'-1}) f_{4}=0 
\end{equation}
\noindent and therefore 
$$(Y^{v} - X_{0}^{\mu}X_{1}X_{2}^{q-q'-1}) f_{4}\in\langle X_{1},X_{2}\rangle .$$  
\noindent We get
$f_{4}\in\langle X_{1},X_{2}\rangle$, since the polynomials 
$X_{1},X_{2}, Y^{v} - X_{0}^{\mu}X_{1}X_{2}^{q-q'-1}$ form a regular sequence in $R$. 
We may write
$$f_{4}= X_{1}p_{1}+X_{2}p_{2},$$
\noindent where $p_{1}$ and $p_{2}$ are polynomials in $R$. Plugging this value in (7.1) we get 
$$X_{1}(f_{1}-Y^{v} + X_{0}^{\mu}X_{1}X_{2}^{q-q'-1})= X_{2}[f_{3}+(Y^{v} - X_{0}^{\mu}X_{1}X_{2}^{q-q'-1})p_{2}],$$ 
\noindent which implies that $X_{1}\mid f_{3}+(Y^{v} - X_{0}^{\mu}X_{1}X_{2}^{q-q'-1})p_{2}$ and 
$X_{2}\mid f_{1}-(Y^{v} - X_{0}^{\mu}X_{1}X_{2}^{q-q'-1})p_{1}$, since $R$ is a UFD. Hence, we obtain 
\begin{eqnarray*}
f_{1} & = & (Y^{v} - X_{0}^{\mu}X_{1}X_{2}^{q-q'-1})p_{1}+ X_{2}p_{3}\\
f_{3} & = & -(Y^{v} - X_{0}^{\mu}X_{1}X_{2}^{q-q'-1})p_{2}+ X_{1}p_{3}
\end{eqnarray*}

Multiplying the third row of $B$ with the column vector gives 
\begin{equation}
-X_{0}f_{1}+ X_{1}f_{3}-(Y^{v} - X_{0}^{\mu}X_{1}X_{2}^{q-q'-1})f_{5}=0
\end{equation}
\smallskip

\noindent Plugging in the values of $f_{1}$ and $f_{3}$ in (7.2) we get,
$$(Y^{v} - X_{0}^{\mu}X_{1}X_{2}^{q-q'-1})(f_{5}+X_{0}p_{1}+X_{1}p_{2})= (X_{1}^2-X_{0}X_{2})p_{3}$$
\noindent and therefore
$$f_{5}=-X_{0}p_{1}-X_{1}p_{2}+(X_{1}^2-X_{0}X_{2})p_{4}$$ and 
$$p_{3}= (Y^{v} - X_{0}^{\mu}X_{1}X_{2}^{q-q'-1}) p_{4}.$$

Multiplying the first row of $B$ with the column vector gives 

\begin{equation}
-X_{2}^{q}f_{1}+ X_{0}^{\lambda-1}Y^{w}f_{3}-(Y^{v} - X_{0}^{\mu}X_{1}X_{2}^{q-q'-1})f_{2}=0
\end{equation}
\smallskip

\noindent Plugging in the values of $f_{1}$, $f_{3}$ and $p_{3}$ in the above obtained equation we get 
$$f_{2}= -X_{2}^{q}p_{1}-X_{0}^{\lambda -1}Y^{w}p_{2}-p_{4}(X_{2}^{q+1} - X_{0}^{\lambda -1}X_{1}Y^{w}).$$
Therefore, 
$${\scriptsize\left[\begin{matrix}
 Y^{v} - X_{0}^{\mu}X_{1}X_{2}^{q-q'-1} & 0 & X_{2}(Y^{v}  - X_{0}^{\mu}X_{1}X_{2}^{q-q'-1})\\
-X_{2}^{q} & -X_{0}^{\lambda-1}Y^{w} & -(X_{2}^{q+1} - X_{0}^{\lambda - 1}X_{1}Y^{w})\\
0 & -(Y^{v} - X_{0}^{\mu}X_{1}X_{2}^{q-q'-1}) & X_{1}( Y^{v} - X_{0}^{\mu}X_{1}X_{2}^{q-q'-1})\\
X_{1} & X_{2} & 0\\
-X_{0} & -X_{1} & X_{1}^2-X_{0}X_{2}
\end{matrix}\right]}
{\scriptsize\left[
\begin{matrix}
p_{1}\\p_{2}\\p_{4}
\end{matrix}\right]}
=
{\scriptsize\left[
\begin{matrix}
f_{1}\\f_{2}\\f_{3}\\f_{4}\\f_{5}
\end{matrix}\right]}$$
Hence the matrix 
$${\scriptsize\left[\begin{matrix}
Y^{v} - X_{0}^{\mu}X_{1}X_{2}^{q-q'-1} & 0 & X_{2}(Y^{v} - X_{0}^{\mu}X_{1}X_{2}^{q-q'-1})\\
-X_{2}^{q} & -X_{0}^{\lambda-1}Y^{w} & -(X_{2}^{q+1} - X_{0}^{\lambda - 1}X_{1}Y^{w})\\
0 & -(Y^{v} - X_{0}^{\mu}X_{1}X_{2}^{q-q'-1}) & X_{1}(Y^{v} - X_{0}^{\mu}X_{1}X_{2}^{q-q'-1})\\
X_{1} & X_{2} & 0\\
-X_{0} & -X_{1} & X_{1}^2-X_{0}X_{2}
\end{matrix}\right]}$$
is the second syzygy matrix. Here, $L_{3}=X_{2}\cdot L_{1} - X_{1}\cdot L_{2}$. Therefore, 
the reduced second syzygy matrix takes the form  
$$C={\scriptsize\left[
\begin{matrix}
Y^{v} - X_{0}^{\mu}X_{1}X_{2}^{q-q'-1} & 0\\
-X_{2}^{q} & -X_{0}^{\lambda-1}Y^{w}\\
0 &-(Y^{v} - X_{0}^{\mu}X_{1}X_{2}^{q-q'-1})\\
X_{1} & X_{2}\\
-X_{0} & -X_{1}
\end{matrix}\right]}.$$

\noindent Hence, the total Betti numbers are $[4,5,2]$. 
\bigskip

\subsection{\textbf{Case (b): $r=r'=1$}.} 

In this case 
\begin{eqnarray*}
\xi_{11} & = & X_{1}^2-X_{0}X_{2}\\
\phi_{0} & = & X_{1}X_{2}^{q} - X_{0}^{\lambda}Y^{w}\\
\phi_{1} & = & X_{2}^{q+1} - X_{0}^{\lambda -1}X_{1}Y^{w}\\
\theta & = & Y^{v} - X_{0}^{\mu}X_{2}^{q-q'}\\
\end{eqnarray*}
\noindent Note that the only difference from the previous case is in the expression of the generator 
$\theta$. It turns out that almost the same computation as above gives us the minimal free resolution 
in this case as well. The syzygy matrices in their minimal form are 
$$A=[\xi_{11},\phi_{0},\phi_{1},\theta],$$
$$B={\scriptsize\left[
\begin{matrix}
-X_{2}^{q} & X_{0}^{\lambda-1}Y^{w} & X_{0}^{\mu}X_{2}^{q-q'}-Y^{v} & 0 & 0\\
X_{1} & -X_{2} & 0 & X_{0}^{\mu}X_{2}^{q-q'}-Y^{v} & 0\\
-X_{0} & X_{1} & 0 & 0 & X_{0}^{\mu}X_{2}^{q-q'}-Y^{v}\\
0 & 0 & X_{1}^2-X_{0}X_{2} & X_{1}X_{2}^{q}-X_{0}^{\lambda}Y^{w} & X_{2}^{q+1}-X_{0}^{\lambda-1}X_{1}Y^{w}
\end{matrix}\right]} \quad {\rm and}$$
$$C={\scriptsize\left[
\begin{matrix}
Y^{v} - X_{0}^{\mu}X_{2}^{q-q'} & 0\\
0 &-(Y^{v} - X_{0}^{\mu}X_{2}^{q-q'})\\
-X_{2}^{q} & -X_{0}^{\lambda-1}Y^{w}\\
X_{1} & X_{2}\\
-X_{0} & -X_{1}
\end{matrix}\right]}.$$
\noindent Hence, the total Betti numbers are $[4,5,2]$. 
\bigskip

\subsection{\textbf{$W = \emptyset$; $r=2$}}

In this case the prime ideal $\mathfrak{p}$ is minimally generated by the set $\{\xi_{11},\phi_{0},\theta\}$. Therefore, 
$\mathfrak{p}$ is a complete intersection and therefore it is minimally resolved by the Koszul complex. The 
syzygy matrices in their minimal forms are  
$$A = [\xi_{11},\phi_{0},\theta],$$
$$B={\scriptsize \left[
\begin{matrix}
-\phi_{0} & -\theta & 0\\
\xi_{11} & 0 & -\theta\\
0 & \xi_{11} & \phi_{0}\\
\end{matrix}
\right]}$$
$$C={\scriptsize\left[
\begin{matrix}
\theta\\
-\phi_{0}\\
\xi_{11}
\end{matrix}
\right]}$$
\noindent Hence, the total Betti numbers are $[3,3,1]$. 

\section{Graded Betti numbers \& Hilbert Functions}
Let $[a,b,c]$ denote the total Betti numbers of $k[\Gamma]$. A graded free resolution of $k[\Gamma]$ as a module over $R$ is given by 
$$0\longrightarrow\oplus_{i=1}^{c}R(-q_{i})\longrightarrow \oplus_{i=1}^{b}R(-p_{i})\longrightarrow \oplus_{i=1}^{a}R(-s_{i})\longrightarrow R\longrightarrow k[\Gamma] \longrightarrow 0.$$
We can easily read the values of $s_{i}$,$p_{i}$ and $q_{i}$ in various cases from the computations that we carried out in the 
previous sections. This information can be used for writing down the Hilbert function of the $R$-module $k[\Gamma]$, if we use 
Theorem 16.2 in \cite{peeva}.
\bigskip

\noindent{\textbf{Case (i): $W\neq\emptyset$; $r=1$, $r' = 2$, $\mu= 0$, $q-q' \neq 1$.}}
\medskip

\begin{tabular}{|c|c|c|}
\hline
$s_{1}=2m_{1}$ & $p_{1}=qm_{2}+2m_{1}$ & $q_{1}=(\lambda-1)m_{0}+vn+2m_{1}$\\
$s_{2}=(q+1)m_{2}$ & $p_{2}=\lambda m_{0}+2m_{1}$ & $q_{2}=m_{1}+\lambda m_{0}+vn$\\
$s_{3}=\lambda m_{0}$ & $p_{3}=vn+2m_{1}$ & $q_{3}=3m_{1}+qm_{2}$\\
$s_{4}=vn$ & $ p_{4}=wn+(\lambda-1)m_{0}+2m_{1}$ & {}\\
{} & $p_{5}=\lambda m_{0}+vn$ & {}\\
{} & $p_{6}=(\lambda-1)m_{0} +(q-q^{'}-1)m_{2}+2m_{1}$ & {}\\
\hline
4 & 6 & 3\\
\hline
\end{tabular}
\bigskip

\noindent{\textbf{Case (ii): $W\neq\emptyset$; $r=1$, $r' = 2$, $\mu\neq 0$, $q-q' = 1$, $\lambda \neq 1$.}}
\medskip

\begin{tabular}{|c|c|c|}
\hline
$s_{1}=2m_{1}$ & $p_{1}=qm_{2}+2m_{1}$ & $q_{1}=(\lambda -1)m_{0}+vn+ 2m_{1}$\\
$s_{2}=m_{1}+qm_{2}$ & $p_{2}=vn+2m_{1}$ & $q_{2}=3m_{1}+\mu m_{0}+qm_{2}$\\
$s_{3}=(q+1)m_{2}$ & $p_{3}=wn+(\lambda-1)m_{0}+2m_{1}$ & {}\\
$s_{4}=(\lambda+\mu) m_{0}$ & $ p_{4}=\lambda m_{0}+vn$ & {}\\
$s_{5}=vn$ & $p_{5}=(v-w)n+(q+1)m_{2}$ & {}\\
{} & $p_{6}=\mu m_{0}+ m_{1}+qm_{2}$ & {}\\
\hline
5 & 6 & 2\\
\hline
\end{tabular}
\bigskip

\noindent{\textbf{Case (iii): $W\neq\emptyset$; $r=1$, $r' = 2$, $\mu\neq 0$, $q-q' \neq 1$, $\lambda = 1$.}}
\medskip

\begin{tabular}{|c|c|c|}
\hline
$s_{1}=2m_{1}$ & $p_{1}=qm_{2}+2m_{1}$ & $q_{1}=(v-w)n+qm_{2}+2m_{1}$\\
$s_{2}=m_{1}+qm_{2}$ & $p_{2}=(\mu+1) m_{0}+2m_{1}$ & $q_{2}=nw+ (\mu+1) m_{0}+2m_{1}$\\
$s_{3}=(q+1)m_{2}$ & $p_{3}=m_{2}+m_{1}+qm_{2}$ & {}\\
$s_{4}=(\mu+1) m_{0}$ & $ p_{4}= m_{0}+vn$ & {}\\
$s_{5}=vn$ & $p_{5}=(q-q^{'})m_{2}+(\mu+1) m_{0}$ & {}\\
{} & $p_{6}=\mu m_{0}+ m_{1}+qm_{2}$ & {}\\
\hline
5 & 6 & 2\\
\hline
\end{tabular}
\bigskip

\noindent{\textbf{Case (iv): $W\neq\emptyset$; $r=1$, $r' = 2$, $\mu\neq 0$, $q-q' = 1$, $\lambda = 1$.}}
\medskip

\begin{tabular}{|c|c|c|}
\hline
$s_{1}=2m_{1}$ & $p_{1}=qm_{2}+2m_{1}$ & $q_{1}=vn+qm_{2}+2m_{1}$\\
$s_{2}=m_{1}+qm_{2}$ & $p_{2}=wn+2m_{1}$ & {}\\
$s_{3}=(q+1)m_{2}$ & $p_{3}= m_{0}+vn$ & {}\\
$s_{4}=(\mu+1) m_{0}$ & $p_{4}=\mu m_{0} +2m_{1}$ & {}\\
$s_{5}=vn$ & $p_{5}=(q^{'}+1)m_{2}+vn $ & {}\\
\hline
5 & 5 & 1\\
\hline
\end{tabular}
\bigskip

\noindent{\textbf{Case (v): $W\neq\emptyset$; $r=1$, $r' = 2$, $\mu\neq 0$, $q-q' \neq 1$, $\lambda \neq 1$.}}
\medskip

\begin{tabular}{|c|c|c|}
\hline
$s_{1}=2m_{1}$ & $p_{1}=qm_{2}+2m_{1}$ & $q_{1}=(\lambda-1)m_{0}+vn+2m_{1}$\\
$s_{2}=m_{1}+qm_{2}$ & $p_{2}=(\lambda+\mu) m_{0}+2m_{1}$ & $q_{2}=(v-w)n+qm_{2}+2m_{1}$\\
$s_{3}=(q+1)m_{2}$ & $p_{3}=vn+2m_{1}$ & $q_{3}=3m_{1}+\mu m_{0}+qm_{2}$\\
$s_{4}=(\lambda+\mu) m_{0}$ & $p_{4}=wn+(\lambda-1) m_{0}+2m_{1}$ & {}\\
$s_{5}=vn$ & $p_{5}=(v-w)n+m_{1}+qm_{2}$ & {}\\
{} & $p_{6}=(\lambda+\mu-1)m_{0}+(q-q^{'}-1) m_{2}+2m_{1}$ & {}\\
{} & $p_{7}=(q^{'}+1)m_{2}+nv$ & {}\\
\hline
5 & 7 & 3\\
\hline
\end{tabular}
\bigskip

\noindent{\textbf{Case (vi): $W\neq\emptyset$; $r=r'=1$, $\mu = 0$.}}
\medskip

\begin{tabular}{|c|c|c|}
\hline
$s_{1}=2m_{1}$ & $p_{1}=m_{1}+\lambda m_{0}$ & $q_{1}=m_{2}+ nv+ \lambda m_{0}$\\
$s_{2}=m_{0}\lambda$ & $p_{2}=nv+ 2m_{1}$ & $q_{2}=m_{1}+nv+ \lambda m_{0}$ \\
$s_{3}=(\lambda -1)m_{0}+m_{1}$ & $ p_{3}= nv+ \lambda m_{0}$ & {}\\
$s_{4}=vn$ & $p_{4}=(\lambda-1)m_{0}+m_{1}+nv$ & {}\\
{} & $p_{5}=\lambda m_{0}+m_{2}$ & {}\\
\hline
4 & 5 & 2\\
\hline
\end{tabular}
\bigskip

\noindent{\textbf{Case (vii): $W\neq\emptyset$; $r=r'=1$, $\mu \neq 0$, $\lambda = 1$.}}
\medskip

\begin{tabular}{|c|c|c|}
\hline
$s_{1}=2m_{1}$ & $p_{1}=qm_{2}+2m_{1}$ & $q_{1}=\mu m_{0}+ m_{1}+(q+1)m_{2}$\\
$s_{2}=m_{1}+qm_{2}$ & $p_{2}=m_{1}+m_{0}(\mu +1)$ & $q_{2}=m_{1}+  m_{0}+vn$ \\
$s_{3}=(q+1)m_{2}$ & $p_{3}=m_{1}+(q+1)m_{2}$ & $q_{3}=\mu m_{0}+ qm_{2}+2m_{1}$\\
$s_{4}=(\mu+1) m_{0}$  & $p_{4}= m_{0}+vn$ & {}\\
$s_{5}=\mu m_{0}+m_{1}$ & $p_{5}=(q+1)m_{2}+n(v-w)$ & {}\\
$s_{6}=vn$ & $p_{6}=(\mu+1) m_{0}+m_{2}$ & {}\\
{} & $p_{7}=m_{0}\mu+ m_{1}+qm_{2}$ & {}\\
{} & $p_{8}=(q+1)m_{2}+m_{0}\mu$ & {}\\
\hline
6 & 8 & 3\\
\hline
\end{tabular}
\bigskip

\noindent{\textbf{Case (viii): $W\neq\emptyset$; $r=r'=1$, $\mu \neq 0$, $\lambda \neq 1$, $q' = 0$.}}
\medskip

\begin{tabular}{|c|c|c|}
\hline
$s_{1}=2m_{1}$ & $p_{1}=qm_{2}+ 2m_{1} $ & $q_{1}=m_{2}+ (v-w)n+m_{1}+qm_{2}$ \\
$s_{2}=m_{1}+qm_{2}$ & $p_{2}=m_{1}+(\lambda+\mu) m_{0} $ & $q_{2}=\mu m_{0}+ m_{1}+(q+1)m_{2}$ \\
$s_{3}=(q+1)m_{2}$ & $p_{3}=m_{1}+(q+1)m_{2}$ & $q_{3}=(v-w)n+ qm_{2}+2m_{1}$\\
$s_{4}=(\mu+1)m_{0}$ & $p_{4}=(v-w)n+m_{1}+qm_{2}$ & {}\\
$s_{5}=(\mu)m_{0}+m_{1}$ & $p_{5}=(v-w)n+(q+1)m_{2}$ & {}\\
$s_{6}=vn$  & $p_{6}=m_{2}+ (\mu+1) m_{0}$ & {}\\
{} & $p_{7}=\mu m_{0}+ m_{1}+qm_{2}$ & {}\\
{} & $p_{8}=\mu m_{0}+ (q+1)m_{2}$ & {}\\
\hline
6 & 8 & 3\\
\hline
\end{tabular}
\bigskip

\noindent{\textbf{Case (ix): $W\neq\emptyset$; $r=r'=1$, $\mu \neq 0$, $\lambda \neq 1$, $q' \neq 0$.}}
\medskip

\begin{tabular}{|c|c|c|}
\hline
$s_{1}=2m_{1}$ & $p_{1}=qm_{2}+ 2m_{1}$ & $q_{1}= (\lambda -1)m_{0}+ 2m_{1}+ vn$ \\
$s_{2}=m_{1}+qm_{2}$ & $p_{2}=m_{1}+ (\lambda+\mu) m_{0}$ & $q_{2}=\mu m_{0}+ m_{1}+ (q+1)m_{2}$ \\
$s_{3}=(q+1)m_{2}$ & $p_{3}= 2m_{1}+ vn$ & $q_{3}=m_{1}+\lambda m_{0}+ vn $\\
$s_{4}=(\lambda+\mu) m_{0}$ & $p_{4}=m_{1}+ (q+1)m_{2} $ & $q_{4}=\mu m_{0}+ qm_{2}+ 2m_{1}$\\
$s_{5}=(\lambda+\mu -1)m_{0}+m_{1}$ & $p_{5}=\lambda m_{0}+ vn$ & {}\\
$s_{6}=vn$ & $p_{6}=(v-w)n+ (q+1)m_{2}$ & {}\\
{} & $p_{7}=m_{2}+ (\lambda+\mu) m_{0}$ & {}\\
{} & $p_{8}=\mu m_{0}+ m_{1}+qm_{2}$ & {}\\
{} & $p_{9}=\mu m_{0}+(q+1)m_{2}$ & {}\\
\hline
6 & 9 & 4\\
\hline
\end{tabular}
\bigskip

\noindent{\textbf{Case (x): $W\neq\emptyset$; $r=2$, $r'=1$, $\mu = 0$, $q'\neq 0$.}}
\medskip

\begin{tabular}{|c|c|c|}
\hline
$s_{1}=2m_{1}$ & $p_{1}=q^{'}m_{2}+(v-w)n+ 2m_{1}$ & $q_{1}=wn+ q^{'}m_{2}+(v-w)n+ 2m_{1}$ \\
$s_{2}=(q+1)m_{2}$ & $p_{2}=nv+2m_{1}$ & $q_{2}=\lambda m_{0}+ nv+2m_{1}$ \\
$s_{3}=(\lambda+1) m_{0}$ & $ p_{3}=\lambda m_{0}+nv$ & {}\\
$s_{4}=\lambda m_{0}+m_{1}$ & $p_{4}=\lambda m_{0}+2m_{1}$ & {}\\
$s_{5}=nv$ & $p_{5}=m_{2}q+2m_{1}$ & {}\\
{} & $p_{6}=m_{1}+m_{0}+(q+1)m_{2}$ & {}\\
\hline
5 & 6 & 2\\
\hline
\end{tabular}
\bigskip

\noindent{\textbf{Case (xi): $W\neq\emptyset$; $r=2$, $r'=1$, $\mu = 0$, $q' = 0$.}}
\medskip

\begin{tabular}{|c|c|c|}
\hline
$s_{1}=2m_{1}$ & $p_{1}= (\lambda+1) m_{0}+2m_{1}$ & $q_{1}=\lambda m_{0}+nw+(\lambda +1)m_{0} +2m_{1}$ \\
$s_{2}=(q+1)m_{2}$ & $p_{2}=m_{0}\lambda+nv$ & {}\\
$s_{3}=(\lambda+1) m_{0}+m_{1}$  & $p_{3}=\lambda m_{0}+2m_{1}$ & {}\\
$s_{4}=\lambda m_{0}+m_{1}$ & $p_{4}=m_{2}q+2m_{1}$ & {}\\
$s_{5}=nv$ & $p_{5}=m_{0}+m_{1}+(q+1)m_{2}$ & {}\\
\hline
5 & 5 & 1\\
\hline
\end{tabular}
\bigskip

\noindent{\textbf{Case (xii): $W\neq\emptyset$; $r=2$, $r'=1$, $\mu \neq 0$, $q'\neq q$, $q' = 0$.}}
\medskip

\begin{tabular}{|c|c|c|}
\hline
$s_{1}=2m_{1}$ & $p_{1}=2m_{1}+(q+1)m_{2}$ & $q_{1}=m_{0}\mu + 2m_{1}+(q+1)m_{2}$\\
$s_{2}=(q+1)m_{2}$ & $p_{2}=m_{1}+ (\lambda+\mu+1) m_{0}+m_{1}$ & $q_{2}=2m_{1}+ \lambda m_{0}+nv$ \\
$s_{3}=(\lambda+\mu+1) m_{0}+m_{1}$ & $ p_{3}=\lambda m_{0}+nv$ & {}\\
$s_{4}=(\lambda+\mu)m_{0}+m_{1}$ & $p_{4}=(\lambda+\mu ) m_{0}+2m_{1}$ & {}\\
$s_{5}=nv$ & $p_{5}=m_{0}\mu+m_{2}q+2m_{1}$ & {}\\
{} & $p_{6}=(q+1)m_{2}+m_{0}+m_{1}$ & {}\\
\hline
5 & 6 & 2\\
\hline
\end{tabular}
\bigskip

\noindent{\textbf{Case (xiii): $W\neq\emptyset$; $r=2$, $r'=1$, $\mu \neq 0$, $q'= q$, $q'\neq 0$.}}
\medskip

\begin{tabular}{|c|c|c|}
\hline
$s_{1}=2m_{1}$ & $p_{1}=(q+1)m_{2}+2m_{1}$&$q_{1}=m_{0}+nv+(q+1)m_{2}$ \\
$s_{2}=(q+1)m_{2}$ & $p_{2}=(\lambda+\mu+1) m_{0}+2m_{1}$&$q_{2}=2m_{1}+\mu m_{0}+(q+1)m_{2}$ \\
$s_{3}=(\lambda+\mu+1)m_{0}+m_{1}$ & $p_{3}=2m_{1}+nv$ & $q_{3}=(v-w)n+2m_{1}+(q+1)m_{2}$\\
$s_{4}=nv$ & $p_{4}=m_{2}+ (\lambda+\mu+1) m_{0}+m_{1}$ & {}\\
{} & $p_{5}=m_{0}\mu+m_{2}q+2m_{1}$ & {}\\
{} & $p_{6}=(q+1)m_{2}+nv$ & {}\\
\hline
4 & 6 & 3\\
\hline
\end{tabular}
\bigskip

\noindent{\textbf{Case (xiv): $W\neq\emptyset$; $r=2$, $r'=1$, $\mu \neq 0$, $q'\neq q$, $q' \neq 0 $.}}
\medskip

\begin{tabular}{|c|c|c|}
\hline
$s_{1}=2m_{1}$ & $p_{1}=(q+1)m_{2}+2m_{1}$ & $q_{1}=m_{0}+(q^{'}+1)m_{2}+nv$ \\
$s_{2}=(q+1)m_{2}$ & $p_{2}=q^{'} m_{2}+(v-w)n+2m_{1}$ & $q_{2}=2m_{1}+\mu m_{0}+(q+1)m_{2}$ \\
$s_{3}=(\lambda+\mu+1) m_{0}+m_{1}$ & $ p_{3}=vn+2m_{1}$ & $q_{3}=(v-w)n+(q+1)m_{2}+2m_{1}$\\
$s_{4}=(\lambda+\mu)m_{0}+m_{1}$ & $p_{4}=\lambda m_{0}+nv$ & {}\\
$s_{5}=nv$ & $p_{5}=(\lambda+\mu)m_{0}+2m_{1}$ & {}\\
{} & $p_{6}=m_{0}\mu+m_{2}q+2m_{1}$ & {}\\
{} & $p_{7}=(q^{'}+1)m_{2}+nv$ & {}\\
\hline
5 & 7 & 3\\
\hline
\end{tabular}
\bigskip

\noindent{\textbf{Case (xv): $W\neq\emptyset$; $r=r'=2$, $\mu=0$.}}
\medskip

\begin{tabular}{|c|c|c|}
\hline
$s_{1}=2m_{1}$ & $p_{1}=2m_{1}+ (q+1)m_{2}+ nv$ & $q_{1}=(q+1)m_{2}+2m_{1}+m_{0}\lambda$ \\
$s_{2}=(q+1)m_{2}$ & $p_{2}=2m_{1}+nv$ & {}\\
$s_{3}=nv$ & $ p_{3}=(q+1)m_{2}+nv$ & {}\\
\hline
3&3&1\\
\hline
\end{tabular}
\bigskip

\noindent{\textbf{Case (xvi): $W\neq\emptyset$; $r=r'=2$, $\mu\neq0$.}}
\medskip

\begin{tabular}{|c|c|c|}
\hline
$s_{1}=2m_{1}$ & $p_{1}=2m_{1}+m_{2}(q+1)$ & $q_{1}=n(v-w)+2m_{1}+m_{2}(q+1)$ \\
$s_{2}=(q+1)m_{2}$ & $p_{2}=2m_{1}+m_{0}(\lambda+\mu)$ & $q_{2}=\mu m_{0}+2m_{1}+m_{2}(q+1)$ \\
$s_{3}=\nu m_{0}$ & $ p_{3}=2m_{1}+nv$ & $q_{3}=nv+2m_{1}+m_{0}(\lambda+\mu)$\\
$s_{4}=nv$ & $p_{4}=n(v-w)+m_{2}(q+1)$ & {}\\
{} & $p_{5}=m_{0}\mu+m_{2}(q+1)$ & {}\\
\hline
4 & 5 & 2\\
\hline
\end{tabular}
\bigskip

\noindent{\textbf{Case (xvii): $W=\emptyset$; $r=1$ , $r'=2$.}}
\medskip

\begin{tabular}{|c|c|c|}
\hline
$s_{1}=2m_{1}$ & $p_{1}=2m_{1}+m_{2}q$ & $q_{1}=nv+2m_{1}+m_{2}q$ \\
$s_{2}=m_{1}+qm_{2}$ & $p_{2}=nv+2m_{1}$ & $q_{2}=nv+(\lambda-1)m_{0}+wn+2m_{1}$\\ 
$s_{3}=(q^{'}+1)m_{2}$ & $p_{3}=2m_{1}+nw+(\lambda-1)m_{0}$ & {}\\
$s_{4}=nv$ & $p_{4}=m_{1}+m_{2}q+nv$ & {}\\
{} & $p_{5}=nv+(q^{'}+1)m_{2}$ & {}\\
\hline
4&5&2\\
\hline
\end{tabular}
\bigskip

\noindent{\textbf{Case (xviii): $W=\emptyset$; $r=r'=1$.}}
\medskip

\begin{tabular}{|c|c|c|}
\hline
$s_{1}=2m_{1}$ & $p_{1}=2m_{1}+m_{2}q$ & $q_{1}=nv+2m_{1}+m_{2}q$ \\
$s_{2}=m_{1}+qm_{2}$ & $p_{2}=2m_{1}+nw+(\lambda-1)m_{0}$ & $q_{2}= nv+2m_{1}+nw+(\lambda-1)m_{0}$\\
$s_{3}=(q+1)m_{2}$ & $p_{3}=nv+2m_{1}$ & {}\\
$s_{4}=nv$ & $ p_{4}=m_{2}(q+1)+nv$ & {}\\
{} & $p_{5}=m_{1}+m_{2}q+nv$ & {}\\
\hline
4&5&2\\
\hline
\end{tabular}
\bigskip

\noindent{\textbf{Case (xix): $W=\emptyset$; $r=2$.}}
\medskip

\begin{tabular}{|c|c|c|}
\hline
$s_{1}=2m_{1}$ & $p_{1}=2m_{1}+m_{2}(q+1)$ & $q_{1}=nv+2m_{1}+m_{2}(q+1)$ \\
$s_{2}=(q+1)m_{2}$ & $p_{2}=2m_{1}+nv$ & {}\\ 
$s_{3}=nv$&$ p_{3}=m_{2}(q+1)+nv$&\\
\hline
3&3&1\\
\hline
\end{tabular}

\bibliographystyle{amsalpha}

\end{document}